\documentclass[a4paper,12pt,reqno]{amsart} 
        \usepackage{graphicx}
        \graphicspath{ {images/} }
        \usepackage{amssymb}
        \usepackage{graphicx}
        \usepackage{tikz-cd}
        \usepackage{tikz}
        \usepackage{amsmath}
        \usepackage{dsfont}
        \usepackage{shuffle}
        \usepackage{afterpage}
        \usepackage{setspace}
        \graphicspath{ {/Users/emanuele/Documents} }
        \usepackage[utf8]{inputenc}
        
        \usepackage{mathrsfs}

\usepackage[margin=1in]{geometry}

\usepackage[bookmarks=true,
bookmarksnumbered=true, breaklinks=true,
pdfstartview=FitH, hyperfigures=false,
plainpages=false, naturalnames=true,
colorlinks=true,pagebackref=true,
pdfpagelabels]{hyperref}
        
        \usepackage[all]{xy}
        \usetikzlibrary{knots} 
        \usepackage[lite]{amsrefs}
        \usepackage{bbm}

        \newtheorem{thm}{Theorem}[section]
        \newtheorem{pro}[thm]{Proposition}
        \newtheorem{lem}[thm]{Lemma}
        \newtheorem{cor}[thm]{Corollary}

        \def\commutatif{\ar@{}[rd]|{\circlearrowleft}}
        
        \newcommand{\eq}[1][r]
           {\ar@<-3pt>@{-}[#1]
            \ar@<-1pt>@{}[#1]|<{}="gauche"
            \ar@<+0pt>@{}[#1]|-{}="milieu"
            \ar@<+1pt>@{}[#1]|>{}="droite"
            \ar@/^2pt/@{-}"gauche";"milieu"
            \ar@/_2pt/@{-}"milieu";"droite"}
        
        \def\dar[#1]{\ar@<2pt>[#1]\ar@<-2pt>[#1]}
          \entrymodifiers={!!<0pt,0.7ex>+} 
        
        \newcommand{\bigon}[4][r]{
            \ar@/^1pc/[#1]^{#2}_*=<0.3pt>{}="HAUT"
            \ar@/_1pc/[#1]_{#3}^*=<0.3pt>{}="BAS"
            \ar@{=>} "HAUT";"BAS" ^{#4}
          }
        
        \newcommand{\bigons}[6][r]{  
            \ar@/^2pc/[#1]^{#2}_*=<0.3pt>{}="HAUT"
            \ar@{}    [#1]     ^*=<0.3pt>{}="MILIEUHAUT"
                               _*=<0.3pt>{}="MILIEUBAS"
            \ar[#1]_(0.3){#3}                  
            \ar@/_2pc/[#1]_{#4}^*=<0.3pt>{}="BAS"
            \ar@{=>} "HAUT";"MILIEUHAUT" ^{#5}
            \ar@{=>} "MILIEUBAS";"BAS" ^{#6}
          }

        \theoremstyle{definition}
        \newtheorem{df}[thm]{Definition}

        \theoremstyle{remark}

        \newtheorem{ex}[thm]{Example}
        
        \allowdisplaybreaks

        \newcommand{\inn}{\operatorname{Inn}}
        
        \newcommand{\Z}{\mathbb{Z}}
        \newcommand{\C}{\mathbb{C}}

        \newcommand\rTo{\longrightarrow}

        \newcommand\mto{\longmapsto}

        \newcommand\rRack{\triangleright}
        
        \newcommand{\one}{\mathbf{1}}

        \def\Aut{\operatorname{Aut}}

        \title{On the Representation Theory of Dihedral and Cyclic Quandles}

        \author{Mohamed Elhamdadi} 
        \address{Department of Mathematics, 
        University of South Florida, Tampa, FL 33620} 
        \email{emohamed@math.usf.edu} 
        
        \author{Prasad Senesi} 
        \address{Department of Mathematics, 
        The Catholic University of America,  Washington, DC 20064} 
        \email{senesi@cua.edu} 
        
        \author{Emanuele Zappala} 
        \address{Yale School of Medicine, 
        	Yale University, New Haven, CT 06511} 
        \email{emanuele.zappala@yale.edu}

\begin{document}

        \maketitle

        \begin{abstract}

        Quandle representations are homomorphisms from a quandle to the group of invertible matrices on some vector space taken with the conjugation operation. 
        We study certain families of quandle representations.
        More specifically, we introduce the notion of regular representation for quandles, investigating in detail the regular representations of dihedral quandles and \emph{completely classifying} them. Then, we study representations of cyclic quandles, giving some necessary conditions for irreducibility and providing a complete classification under some restrictions. Moreover, we provide various counterexamples to constructions that hold for group representations, and show to what extent such theory has the same properties of the representation theory of finite groups. In particular, we show that Maschke's theorem does not hold for quandle representations.  
        \end{abstract}
        
        \tableofcontents
        
        \section{Introduction}
        
        Quandles are algebraic objects whose defining axioms algebraically capture the Reidemeister moves from knot theory. Since they have been introduced indepentently by Joyce and Matveev in \cite{Joyce,Matveev}, they have found several applications in low-dimensional topology. The {\it fundamental quandle} of Joyce and Matveev, in fact, is a complete invariant of knots and links (up to mirror image and orientation reversal) and it is therefore a very powerful tool in distinguishing links. However, the fact that it completely classifies links is substantially indicative of the fact that it simply translates the classification complexity of knot theory into another equivalent complexity in algebraic terms. In fact, more practically, the fundamental quandle is derived from a presentation where generators and relations are obtained from a diagram of the knot following a procedure very similar to the Wirtinger presentation of the knot group. Comparing two presentations of two quandles is a significantly hard problem. Consequently, in order to apply these algebraic tools, several weaker invariants derived from the fundamental quandle have been introduced. For example, counting the number of homomorphisms from the fundamental quandle of a link to a fixed quandle $X$ gives a easily comparable quantity that is invariant under Reidemeister moves, which is commonly referred to as {\it quandle coloring invariant} \cite{EN}. The counting invariant is significantly weaker than the fundamental quandle, and it is not difficult to incur into non-isotopic links and given target quandles that produce identical coloring invariants. We have therefore come across an over-simplification of the original problem. In the spectrum of complexity spanning from the fundamental quandle to the quandle coloring invariant one finds several other invariants, such as the {\it cocycle invariant} of links and knotted surfaces in $4$-space introduced in \cite{CJKLS}. The cocycle invariant has in fact been adapted to several other object including, for example, the cocycle invariants of handlebodies \cite{IIJO}.
        
        The epitomizing example of quandle is a group $G$ with operation $*$ defined by conjugation, i.e. $x*y = yxy^{-1}$, where juxtaposition indicates multiplication in the group $G$. This fundamental example gives the concrete idea of using the group of invertible matrices acting over a fixed vector space with quandle operation given by conjugation to represent quandles. This provides a way of introducing a representation theory for quandles (see also \cite{EM}), where a representation of a quandle $X$ is given by a quandle homomorphism from $X$ into $M(n,n)$, where $n$ is the dimension of the representation. Of course, similarly to the notion of regular representation for a group, of the adjoint represenation of a Lie algebra, there is a natural representation that always arises when considering a quandle. This is what is hereby called the {\it regular representation} in analogy with the group theoretic case, and it simply leverages the fact that every element $x$ in a quandle $X$ acts by right multiplication on $X$. Therefore, defining a vector space $V = \mathbbm k\langle X \rangle$ as the span of the elements of $X$, we obtain for each $x$ a linear map $V\longrightarrow V$. This is seen to satisfy the defining condition of quandle representation. The space $\mathbbm k\langle X \rangle$ with addition and a natural multiplication 
        \[
        (\sum_{x \in X}a_x e_x)(\sum_{y \in X}b_ye_y):=\sum_{x,y \in X}a_xb_ye_{x*y},
        \]
        becomes a (non-associative) ring.  Thus in \cite{BPS} a theory of quandle rings was proposed in a parallel to the theory of group rings. In particular it was shown that quandle rings of dihedral quandles are not power-associative.  This was further generalized in \cite{EFT} to quandle rings of all non-trivial quandles. 
 Examples of non-isomorphic quandles with isomorphic quandle rings were given in \cite{EFT} and a decomposition of the quandle ring of dihedral quandles into simple right ideals was given, but the results were incomplete as they did not address the case of dihedral quandles of odd order.  The complete decomposition of  the regular quandle representation of a dihedral quandles was addressed independently in \cite{ESZ} and \cite{KS}.  Recall that The \emph{augmentation} surjective homomorphism $\epsilon: \mathbbm k\langle X \rangle \rightarrow k$ given by $\epsilon (\sum_{x \in X}a_x e_x):= \sum_{x \in X}a_x$ gives the decomposition into two non-trivial subrepresentations
        \[
        \mathbbm k\langle X \rangle = {\rm ker} (\epsilon) \oplus k,
        \]
        where ${\rm ker}(\epsilon) = \left\{ \sum_{x \in X} a_x e_x |\; \sum a_x = 0 \right\}$ is subrepresentation of \emph{codimension} one.

        This theory presents algebraic interest, in that it can be regarded as a theory of representations of generalized conjugation operations. In addition, it has a fundamental knot-theoretic objective, which is the main underlying driving force of the present article. Namely, this is the fact that the representation-theoretic properties of the fundamental quandle of a link do not depend on the diagram that one chooses to define the fundamental quandle. Therefore, one can distinguish links by distinguishing the properties of their fundamental quandles in terms of their quandle representations. Several invariants can be extrapolated from the representation category of the fundamental quandle of a link. We will not indulge on the topological applications of this theory in this article, and defer such work to a subsequent article, while we will focus on the algebraic properties of some classes of representations, and unveil some important examples that further motivate this study.
        
    This article is organized as follows:  Section~\ref{review} reviews the basics of racks and quandles including the definition of cyclic quandles used later in the article.  In Section~\ref{cyclic} we provide a presentation of the Alexander quandles of cyclic type in order to facilitate the representation theory explored later in this paper. Although it is not the focus of this work, our results in this section provide unexpected and previously unknown results: a presentation of an arbitrary cyclic quandle via generators and relations, and the complete classification of all cyclic quandles (Theorem \ref{cyclicclas}). Section~\ref{rep} gives the basic definitions of representation theory for quandles.  In Section~\ref{Rep Dihedral} we introduce the regular representation of dihedral quandles, and completely classify them by providing their explicit decomposition into irreducible subrepresentations (Theorem \ref{dihedral}). In Section~\ref{sec:irreps}, we study representations of cyclic quandles, examining in detail the case of dimension $2$ representations, and describe a class of representations for which a complete classification is given (Theorem \ref{thm:higher_dim_cyclic}). In Section~\ref{sec:non_decomposability} we provide various counterexamples regarding the decomposability of quandle representations, and show that Maschke's theorem does not hold for quandles. Some computational results are deferred to the Appendix.

        \section{Review of racks and quandles}\label{review}

        \begin{df}	
        \cites{EN, Joyce, Matveev} 
        A {\it rack} is a set $X$ provided with a binary operation 
        \[
        \begin{array}{lccc}
        \rRack: & X\times X & \rTo &X \\
         & (x,y) & \mto & x\rRack y
        \end{array}
        \] such that 
        \begin{itemize}
        \item[(i)] for all $x,y\in X$, there is a unique $z\in X$ such that $y=z\rRack x$;
        \item[(ii)]({\it right distributivity}) for all $x,y,z\in X$, we have $(x\rRack y)\rRack z=(x\rRack z)\rRack (y\rRack z)$.
        \end{itemize}
        \end{df}
        
        Observe that property (i) also reads that for any fixed element $x\in X$, the map $R_x:X\ni y\mto y\rRack x\in X$ is a 
        bijection. 
        Also, notice that the distributivity condition is equivalent to the relation $R_x(y\rRack z)=R_x(y)\rRack R_x(z)$ for all $y,z\in X$.

        Unless otherwise stated, we will always assume our racks (or quandles) to be finite racks (or quandles).

        \begin{df}
        A {\it quandle} is a rack such that 
         $x\rRack x=x,\forall x\in X$. 
        \end{df}
       	\begin{ex} 
       The set $\Z_n = \left\{ 1, 2, \ldots, n \right\}$, with quandle operation $x \rRack y = 2y - x$ (\text{mod } $n)$ is a quandle, called \emph{dihedral} quandle.
         \end{ex}
        	
        	\begin{ex}
        	Any group $G$ with the operation $x\rRack y=yxy^{-1}$ is a quandle, called conjugation quandle and denoted by Conj(G). 
        \end{ex}
        \begin{ex}
        	Any group $G$ with the operation $x\rRack y=yx^{-1}y$ is a quandle called the Core quandle of $G$ and denoted Core(G). 
        \end{ex}
        
        \begin{ex}
        	Let $G$ be a group and $f \in {\rm Aut}(G)$, then one can define a quandle structure on $G$ by $x\rRack y=f(xy^{-1}) y $.  It is called a {\it generalized Alexander quandle}.	If $G$ is abelian, the operation becomes $x\rRack y=f(x)+ (Id-f) (y) $, where $Id$ stands for the identity map.  This quandle is called an {\it Alexander quandle}.
        	
        	In particular,  let $q$ be a power of a prime $p$,  and let $\mathbb{F}_q$ be the field with $q$ elements.   For an element $\alpha \in \mathbb{F}_q$, we will denote by $(\mathbb{F}_q,  \alpha)$ the Alexander quandle with $x \rRack y = \alpha x + (1-\alpha) y$ for all $x, y \in \mathbb{F}_q$.  These quandles will be closely studied later in this paper.
        \end{ex}

        For $x \in X$, we denote by $R_x$ the quandle automorphism of $X$ given by 
        \[ R_x (y) = y \rRack x.\]
        As each $R_x$ permutes the elements of $X$, we can consider the subgroup $\left\langle R_x \right\rangle_{x \in X}$ in $S_X$ (the symmetric group on $X$) generated by the $R_x$, which we denote by $\inn(X)$ (the group of inner automorphisms of $X$).  
        
        
        
        \begin{df} \cite{KTW}
        A finite quandle $X$ of cardinality $n>2$ is said to be of cyclic type if for all $i \in X$, the right multiplication $R_i$ is a cycle of length $(n-1)$. 
        \end{df}

        The phrase `quandle of cyclic type' is prevalent in the literature.  In this manuscript, we also use the phrase `cyclic quandle' to mean the same thing.
        
        Here are few examples of \emph{cyclic} quandles.
        \begin{itemize}
        \item
        For $n=3$, the dihedral quandle $\mathbb{Z}_n$ is the only cyclic quandle.
        
        \item
        
        For $n=4$, the Alexander quandle $\mathbb{Z}_2[t]/(t^2+t+1) $ is the only cyclic quandle.  It is isomorphic to the quandle $X=\{1,2,3,4\}$ with $R_1=(234)$, $R_2=(143)$, $R_3=(124)$ and $R_4=(132)$.
        
        \item
        
        For $n=5$, there are exactly two cyclic quandles which are the Alexander quandles $\mathbb{Z}_5[t]/(t-3)$ and $\mathbb{Z}_5[t]/(t-2)$.
        \end{itemize}
        
    Quandles of cyclic type are (isomorphic to) certain Alexander quandles $(\mathbb{F}_q, \alpha)$.  This correspondence is given in \cite{Wada}:
    
    \begin{thm}[\cite{Wada}]
    A quandle $X$ is of cyclic type if and only if $X$ is isomorphic to an Alexander quandle $(\mathbb{F}_q,  \alpha)$, where $\alpha$ is a primitive element in $\mathbb{F}_q$. 
    \end{thm}
    
 \section{Quandles of cyclic type: a presentation and classification}\label{cyclic}
   To facilitate the representation theory explored later in this paper, we provide in this section a presentation of the Alexander quandles of cyclic type $(\mathbb{F}_q,  \alpha)$.  

Let $p$ be a prime,  $q = p^s$ for some $s \geq 1$, and let $\alpha$ be a primitive element in the finite field $\mathbb{F}_q$ of order $q$.  For any nonzero element $r \in \mathbb{F}_q$, we define $\log_{\alpha} r$ to be the unique integer $0 \leq \log_{\alpha} r \leq q-2$ such that 
    \[ \alpha^{\log_{\alpha} r } = r.\] 

    For elements $x_1, x_2, \ldots, x_s$ in a quandle $X$,  we will denote by $x_1 x_2 \cdots x_s$ the left-associated element $( \cdots( ( x_1 \rRack x_2) \rRack x_3 ) \cdots \rRack x_s)$.   If $S$ is the set of elements $\left\{ x_1, x_2, \ldots, x_s \right\}$, we will say that an element $u \in X$ can be {\it left-associated in $S$} if we can write $u = x_{q_1} x_{q_2} \cdots x_{q_t}$, for $x_{q_i} \in S$.  
    
    Let $F_{q, \alpha}$ be the quandle with presentation 
    \[ F_{q,  \alpha} =  \left\langle x, y \; \; | \; \; ab^{q-1} = a, \; \; ab^k = ba^{\log_{\alpha}(1-\alpha^k)} \right\rangle_{a \neq b \in \left\{ x, y \right\},  \; \; \; 1 \leq k \leq q-2}.\]
\begin{lem}\label{log_relations}
Let $u, v$ be nonzero elements in $\mathbb{F}_{q, \alpha}$, $n \in \mathbb{Z}_+$, and  $s \in \mathbb{Z}_{\geq 0}$.
\end{lem}
    \begin{enumerate}
        \item For any $z \in F_{q ,\alpha}$,  we have $zx^n = zx^{n \; \text{mod } (q-1)}$, and similarly for $zy^n$.
        \item For any nonzero $u \in \mathbb{F}_q$, 
        \[ -\log_{\alpha}(u) \equiv \log_{\alpha}(u^{-1}) \text{ mod } (q-1).\]
        \item For any nonzero $u,v \in \mathbb{F}_q$,
        \[ \log_{\alpha} uv \equiv \log_{\alpha} u + \log_{\alpha} v \text{ mod } (q-1).\]
        \item  For any $b, u \in \mathbb{F}_q$, $u \neq 0$,   \[
            b^{\log_{\alpha} u} b^s = b^{\log_{\alpha}( u \alpha^s)}.\]
        \item      For any $u \in \mathbb{F}_q \backslash \left\{ 0, 1 \right\}$,   \[
             xy^{\log_{\alpha}(u)} = yx^{\log_{\alpha}(1-u)}.\]
            \end{enumerate}      
\begin{proof}
Parts (2) - (4) are standard properties of $\log_{\alpha}$ over a finite field, and (5) follows since any element $u \in \mathbb{F}_q$ can be written as $\alpha^k$.  Only (1) requires proof.  For this argument we need a well--defined notion of the length of an element in $F_{q, \alpha}$, given as follows.

  We can define the set $\mathcal{A}$ of admissible words in the symbols 
$x, y, \rRack, \left( \right. $ and $\left. \right) $ inductively: we have $x, y \in \mathcal{A}$, and for $u, v \in \mathcal{A}$, $(u) \rRack (v) \in \mathcal{A}$.  Furthermore we write $(x) = x$ and $(y) = y$ for any occurrence of $(x)$ or $(y)$ in a word. Elements of $F_{q,\alpha}$ consist of equivalence classes of admissible words modulo the defining relations. For an element $u \in F_{q, \alpha}$, we define its length $\ell(u)$ to be the smallest integer $n$ for which there exists a sequence $\left\{ s_i \right\}_{i=1}^n$, $s_i \in \left\{ x, y \right\}$, such that $u$ has a representative written with $\left\{ s_i \right\}_{i=1}^n$.  For example, if $q = 5$ and $u = y((xy^4)x) \in F_{q, \alpha}$, the length of $u$ is 2, since $y((xy^4)x) = y(xx) = yx$.  
  
The argument for (1) is brief.  For any $u \in F_{q, \alpha}$, a short inductive argument on $k$ shows that $uxy^{k} = (uy^k)(xy^k)$ for any positive integer $k$.  Therefore $(ux)y^{q-1} = (uy^{q-1})(xy^{q-1})$, and another inductive argument on the length of $u$ gives us the result in (1).
\end{proof}

    We note several consequences of $\log_{\alpha}$ and the given relations.

    \begin{lem}\label{lem:generators}
    Any element of $F_{q, \alpha}$ can be  left-associated  in $\left\{ x, y \right\}$, the set containing the generators.
    \end{lem}
    \begin{proof}
    Any quandle $X$ with operation $\rRack:  X \times X \rightarrow X$ also comes equipped with a second operation $\rRack^{-1}: X \times X \rightarrow X$, the \textit{right inverse} of $\rRack$, satisfying $(a \rRack b) \rRack^{-1} b = a$ for all $a, b \in X$.   If $X$ is generated by elements $\left\{ x_i \right\}_{i=1}^n$, then any element of $X$ can be left-associated in the generators $\left\{x_i\right\}$ if we allow both operations $\rRack$ and $\rRack^{-1}$ \cite{Winkler}; i.e., for any $a \in X$ We can write
    \[ a = (( \ldots ( x_{i_1} \rRack^{\pm 1} x_{i_2} ) \rRack^{\pm 1} x_{i_3} ) \ldots ) \rRack^{\pm 1} x_{i_k} \]
for appropriate choices of $\rRack$, $\rRack^{-1}$ and generators $x_{i_j}$.

Furthermore, for the quandle $X = F_{q, \alpha}$, the definining relations and Lemma 3.1 give us $a \rRack^{-1} b = ab^{q-2}$ for any $a, b \in F_{q, \alpha}$.  The conclusion follows. 
        \end{proof}
   \begin{lem}\label{relns2}
    Let $p$ be a prime, $q = p^n$, and $\alpha$ a primitive element in the field $\mathbb{F}_q$.  For nonnegative integers $r$ and $s$ with $r+s > 0$, set $\mu(r,s) = \alpha^{r+1} - \alpha^{s+1} + \alpha^s$.  In the quandle $F_{q, \alpha}$, we have the relations 
    \[ (xy^r)(xy^s) = \begin{cases}
    y,  &  \mu(r,s) = 0  \\
     xy^{\log_{\alpha}(\mu(r,s))},  & \mu(r,s) \neq  0. \end{cases} \]
    \end{lem}
    \begin{proof}
    Assume $\mu(r,s) =0$.  Then $0 = \alpha^{r+1} - \alpha^{s+1} - \alpha^s$, so $\alpha - \alpha^{r-s+1} = 1$. This also gives us $r-s \not \equiv 0 \mod (q-1)$ (else $\alpha -\alpha^{r-s+1} = \alpha - \alpha = 0$) and so $xy^{r-s} = yx^{\log_{\alpha}(1 - \alpha^{r-s})}$. 
    
 A brief inductive argument shows that $(xy^r)(xy^s) = (xy^{r-s}x)y^s$.  Let $m$ be any nonnegative integer such that $m(q-1)+r > s$.  Then since $xy^{q-1} = x$, we can write 
\begin{eqnarray*}
(xy^r)(xy^s) &=& (xy^{m(q-1) + r})(xy^s)\\
&=& (xy^{m(q-1) + r - s})(x)y^s\\
&=& yx^{\log_{\alpha}(1 - \alpha^{m(q-1)+r-s})\alpha}y^s\\
&=& yx^{\log_{\alpha}(1 - \alpha^{r-s}) \alpha} y^s\\
&=& yx^{\log_{\alpha}(\alpha - \alpha^{r-s+1})}y^s\\
&=& yx^{\log_{\alpha}(1)}y^s\\
&=& yx^0y^s = y^{s+1} = y.
\end{eqnarray*}

\pagebreak 

Now suppose $\mu(r,s) \neq 0$.  We prove this case by induction on the sum $r+s$.  We first verify the base case $r+s=1$ in two cases: $r=1$, $s=0$, and $r=0$, $s=1$.  First suppose $r=1$ and $s=0$.  Then 
    \[ (xy^r)(xy^s) = (xy)x=(yx^{log_{\alpha}(1-\alpha)})x = yx^{\log_{\alpha}(\alpha- \alpha^2)} = xy^{\log_{\alpha}(1-\alpha+\alpha^2)}=xy^{\log_{\alpha}(\mu(1,0))},\]
where we have used the assumption $\mu(1,0) \neq 0$ when writing $yx^{\log_{\alpha}(\alpha - \alpha^2)} = xy^{\log_{\alpha}(1 - \alpha + \alpha^2)}$. 

\noindent For $r=0$, $s=1$, we have 
\begin{eqnarray*}
(xy^r)(xy^s) &=& (x)(xy)\\
&=& (xy^{q-1})(xy) \\
&=& ((xy^{q-2})(x))y\\
&=& (yx^{\log_{\alpha}(1 -\alpha^{q-2})\alpha})y\\
&=& (yx^{\log_{\alpha}(\alpha - 1)})y\\
&=& xy^{\log_{\alpha}(1-(\alpha-1))\alpha}\\
&=& xy^{\log_{\alpha}(\alpha - \alpha^2 + \alpha)} = xy^{\log_{\alpha}(\mu(0,1))},
\end{eqnarray*} 
where we have used the assumption $\mu(0,1) \neq 0$ when writing $yx^{\log_{\alpha}(\alpha - 1)} = xy^{\log_{\alpha}(1 - (\alpha -1))}$. 

For the inductive step, we fix $r,s$ with $r+s > 1$, and assume the result is true for all indices $\tilde{r}, \tilde{s}$ with $\tilde{r} + \tilde{s}
 < r+s$.  Then since $\mu(r-1, s-1))\alpha = \mu(r,s)$, we have 
\begin{eqnarray*}
(xy^r)(xy^s) &=& (xy^{r-1})(xy^{s-1})y\\
&=& xy^{\log_{\alpha}(\mu(r-1, s-1)\alpha)}\\
&=& xy^{\log_{\alpha}(\mu(r,s))}.
\end{eqnarray*}
    \end{proof}

    \begin{thm}\label{cyclic_classification}
    Let $p$ be a prime,  $q = p^s$, and $\alpha$ a primitive element in $\mathbb{F}_q$.  Then the Alexander quandle $(\mathbb{F}_q, \alpha)$ is isomorphic to $F_{q, \alpha}$.  
    \end{thm}
    \begin{proof}
    
    In $(\mathbb{F}_q, \alpha)$, the formulas $01^k = 1 - \alpha^k$ and $10^k = \alpha^k$ (for any nonnegative integer $k$) can be easily verified by the definition of the quandle operation in $(\mathbb{F}_q, \alpha)$ and an inductive argument in $k$.  In particular, this shows that $(\mathbb{F}_q, \alpha)$ is generated by $0$ and $1$.  Furthermore, as a consequence of these formulas we have
    \begin{eqnarray*}
    & & 01^{q-1} = 1 - \alpha^{q-1} = 1 - 1 =0,\\
    & & 10^{q-1} = \alpha^{q-1} = 1, 
    \end{eqnarray*}
    and for any integer $1 \leq k \leq q-2$, 
    \begin{eqnarray*}
    & & 10^{\log_{\alpha}(1-\alpha^k)} = \alpha^{\log_{\alpha}(1-\alpha^k)} = 1 - \alpha^k = 01^k,\\
    & & 01^{\log_{\alpha}(1-\alpha^k)} = 1 - \alpha^{\log_{\alpha}(1-\alpha^k)} = 1 - (1 - \alpha^k) = \alpha^k = 10^k.
    \end{eqnarray*}
This shows that $(\mathbb{F}_q, \alpha)$ is generated by two elements ($0$ and $1$) which satisfy the defining relations of $F_{q, \alpha}$, from which it follows that there exists a quandle epimorphism $F_{q, \alpha} \twoheadrightarrow (\mathbb{F}_q, \alpha)$. Now the result is proven once we prove $|F_{q, \alpha}| \leq |(\mathbb{F}_q, \alpha)|$.

      Consider the following set of elements in $F_{q, \alpha}$: 
    \[ S = \left\{  x, y,  xy^r \right\}_{1 \leq r \leq q-2}.\]
    We have $|S| \leq q$.  Now we will show that any other element in $F_{q, \alpha}$ is in $S$ as well.  By Lemma \ref{lem:generators}, any element in $F_{q, \alpha}$ can be left-associated in $\left\{ x,y \right\}$.  We proceed by induction on the number $k$ of $x$ and $y$ used to write such a left-associated element.   For $k=1$ and $2$, the statement is true using the relations in $F_{q, \alpha}$.    Now assume the statement is true for elements of length $\leq k$,  left-associated in $\left\{ x, y \right\}$, where $k \geq 3$.   Let $u$ be an element of length $k+1$,  left-associated in $\left\{ x, y \right\}$: 
    \[ u = r_{1} r_2 \cdots r_k r_{k+1}, \; \; \; r_i \in \left\{ x, y \right\}.\]
    By induction,  the element $ r_{1} r_2 \cdots r_k$ is equal to some element in $S$.  So the proof is complete once we show that $S$ is closed under `right multiplication' by $x$ or $y$.  This is easily seen for $y$.  For right multiplication by $x$, the only nontrivial observation is 
    \begin{eqnarray*}
    xy^rx &=& (xy^r)x \\
    &=& (yx^{\log_{\alpha}(1 - \alpha^r)}) x \\
    &=& yx^{\log_{\alpha}(1 - \alpha^r) + 1}\\
    &=& yx^{(\log_{\alpha}(1 - \alpha^r) + 1) \text{mod} (q-1) } \in S. \\
    \end{eqnarray*}
    Therefore $| F_{q,  \alpha} | \leq q$, and the isomorphism is proven.
    \end{proof}

\subsection{A complete classification of quandles of cyclic type}   The presentation of cyclic quandles given in Theorem \ref{cyclic_classification} is used in subsequent sections to prove some representation-theoretic results.  Although the remainder of this Section is not needed for any of the representation theory in the remainder of this paper, Theorem \ref{cyclic_classification} also provides us with a previously unknown result: the complete classification of quandles of cyclic type.  If $\alpha$ and $\beta$ are distinct primitive elements in the field $\mathbb{F}_q$, there was no way to determine if $(\mathbb{F}_q, \alpha)$ and $(\mathbb{F}_q, \beta)$ are isomorphic quandles.   In Theorem \ref{cyclicclas} below, we obtain a necessary and sufficient condition for this isomorphism, and also provide a simple formula to count the number of nonisomorphic quandles of a given order.  
    
    \begin{lem}\label{product1}  Let $m, n, k,$ and $r$ be nonnegative integers, $1 \leq r < q-1$, and $u, v \in F_{q, \alpha}$.  Then 
    \begin{enumerate} 
    \item $(xy^m)(xy^n)^k = xy^{\log_{\alpha}(\alpha^{m+k} - \alpha^{n+k} + \alpha^n) }$, 
    \item $uv^{r} = vu^{\log_{\alpha}(1 - \alpha^r)}$, 
    \item If $u \neq v$, then $uv^k = uv^{\ell}$ if and only if $ k \equiv \ell \text{ mod } (q-1)$.
    \end{enumerate}
    \end{lem}
    \begin{proof}
    We prove (1) by induction on $k$.  The base case $k=1$ is given in Lemma \ref{relns2}.  For the inductive step, we have 
    \begin{eqnarray*}
    (xy^m)(xy^n)^{k+1} &=& (xy^m)(xy^n)^{k}(xy^n)\\
    &=& xy^{\log_{\alpha} (\alpha^{m+k} - \alpha^{n+k} + \alpha^n )} (xy^n)\\
    &=& xy^{\log_{\alpha} \mu(\log_{\alpha} (\alpha^{m+k} - \alpha^{n+k} + \alpha^n ),n )}\\
    &=& xy^{\log_{\alpha} ((\alpha^{m+k} - \alpha^{n+k} + \alpha^n )\alpha - \alpha^{n+1} + \alpha^n )} \\
    &=& xy^{\log_{\alpha}(\alpha^{m+k+1} - \alpha^{n+k+1} + \alpha^n) }.
    \end{eqnarray*}
    
    Part (2) is proven by induction on $r$.  We begin by representing $u$ and $v$ as 
     \[ u = xy^m, \; \; \; v = xy^n \]
     for some $0 \leq m, n \leq q-2$.  For the base case $r = 1$, Lemma \ref{product1}, item (1), gives us 
     \begin{eqnarray*}
     vu^{\log_{\alpha}(1 - \alpha)} &=& (xy^n)(xy^m)^{{\log_{\alpha}(1 - \alpha)} }\\
     &=& xy^{\log_{\alpha} (\alpha^{n + \log_{\alpha}(1-\alpha)} - \alpha^{m + \log_{\alpha}(1-\alpha)} + \alpha^m) }\\
     &=& xy^{\log_{\alpha} (\alpha^n(1-\alpha)-\alpha^m(1-\alpha) + \alpha^m ) }\\
     &=& xy^{\log_{\alpha}(\alpha^{m+1}+ \alpha^n - \alpha^{n+1} ) }\\
     &=&  xy^{\log_{\alpha}( \mu(m, n))} = (xy^m)(xy^n) = uv.
     \end{eqnarray*}
    For the inductive step, we have 
    \begin{eqnarray*}
        uv^{r+1} &=& uv^r v\\
                  &=& vu^{\log_{\alpha} (1 - \alpha^r)}v\\
                  &=& (xy^n)(xy^m)^{\log_{\alpha} (1 - \alpha^r)}xy^n\\
                  &=& xy^{\log_{\alpha}(\alpha^{n+\log_{\alpha} (1 - \alpha^r)} - \alpha^{m +\log_{\alpha} (1 - \alpha^r)} + \alpha^m)}xy^n \\
                  &=& xy^{\log_{\alpha}((1-\alpha^r)\alpha^n - (1 - \alpha^r) \alpha^m + \alpha^m)}xy^n \\
                  &=& xy^{\log_{\alpha} (\mu(\log_{\alpha}((1 - \alpha^r)\alpha^n - (1 - \alpha^r)\alpha^m + \alpha^m),n))}  \; \; \; \;\; \; \; \; \text{ (Lemma \ref{relns2}) }\\
                  &=& xy^{\log_{\alpha}[((1 -\alpha^r)\alpha^n - (1 - \alpha^r)\alpha^m + \alpha^m) \alpha+ \alpha^n -\alpha^{n+1}]}\\
                  &=& xy^{\log_{\alpha}(\alpha^{r+m+1} -\alpha^{r+n+1}  + \alpha^n) }\\
                  &=& xy^{\log_{\alpha}(\alpha^n(1 - \alpha^{r+1}) - \alpha^m(1 - \alpha^{r+1}) + \alpha^m)}\\
                  &=& xy^{\log_{\alpha}(\alpha^{n + \log_{\alpha} (1 - \alpha^{r+1})} - \alpha^{m +\log_{\alpha} (1 - \alpha^{r+1})} + \alpha^m)}\\
                  &=& (xy^n)(xy^m)^{\log_{\alpha}(1- \alpha^{r+1} )} \; \; \; \;\; \; \; \; \text{ (part (1)) }\\
                  &=& vu^{\log_{\alpha}(1- \alpha^{r+1} )}.
    \end{eqnarray*}
    We first prove Part (3) explicitly for $a, b$ in $(\mathbb{F}_q, \alpha)$ (in place of $u, v$ in $F_{q, \alpha}$).  Induction on $k$ shows that $ab^k = \alpha^k a + (1 - \alpha^k)b$. Suppose $ab^k = ab^{\ell}$.   If $a = 0$, then this assumption gives us 
    \[ (1 - \alpha^k) b = 0b^k = 0 b^{\ell} = (1- \alpha^{\ell}) b,\]
    and if $b = 0$ it gives us $\alpha^k a = \alpha^{\ell} a$.  Either way we obtain $\alpha^k = \alpha^{\ell}$, hence $k \equiv \ell \mod (q-1)$.

    If both $a, b  \neq 0$, we argue by contradiction, assuming $k \not \equiv \ell \mod (q-1)$.  Then $\alpha^k \neq \alpha^{\ell}$.  The assumption $ab^k = ab^{\ell}$ gives us 
    \[ \alpha^k a + (1- \alpha^k)b = \alpha^{\ell} a + (1 - \alpha^{\ell})b, \; \; \; \; \text{ so } \; \; \; \; (\alpha^k - \alpha^{\ell})a = (\alpha^k - \alpha^{\ell})b, \]
    and since $\alpha^k - \alpha^{\ell} \neq 0$, we have $a/b = (\alpha^k - \alpha^{\ell}) / (\alpha^k - \alpha^{\ell}) = 1$, which contradicts $a \neq b$.  This proves the result for $a, b$ in $(\mathbb{F}_q, \alpha)$.  Finally we use the isomorphism $F_{q, \alpha} \cong (\mathbb{F}_q, \alpha)$ to obtain the same result in $F_{q, \alpha}$.
    \end{proof}
    
Suppose $\alpha$ and $\beta$ are primitive elements in a finite field $F_q$, where $q = p^n$ for $p$ prime.  If $\alpha $ can be writen as a `prime power' of $\beta$, i.e., $\log_{\beta} \alpha = p^s$ for some integer $s$, then $\alpha = \beta^{p^s}$, and $\beta = \beta^{p^n} = (\beta^{p^s})^{p^{n-s}} = \alpha^{p^{n-s}}$, so that $\beta$ can also be written as a `prime power' of $\alpha$.  Hence this condition is an equivalence relation on the set of primitive elements of $F_q$.  We will say that two primitive elements are \textit{prime power equivalent} if this is the case.

\begin{lem}\label{loglemma}
Let $q=p^n$, p prime, and let $\alpha$, $\beta$ be two primitive elements in $\mathbb{F}_q$.  Then $\alpha$ and $\beta$ are prime-power equivalent if and only if  
\begin{equation}\label{logeqn}
\log_{\alpha}(1-\alpha^k) \equiv \log_{\beta}(1 - \beta^k), \; \; 0 < k < q-1.
\end{equation} 
\end{lem}
\begin{proof}
 Set $N = \log_{\beta} \alpha$.  Then (\ref{logeqn}) is equivalent to 
\[ \left( 1 - \beta^k\right)^N = 1 - \alpha^k, \; \; 0 < k < q-1.\]
Note that $N$ must be odd, for if $N$ were even, then the choice $k = \log_{\beta}2$ would give us 
\[
 1  = (-1)^N 
 = \left( 1 - \beta^k\right)^N 
 = 1 - \alpha^k 
 = 1 - 2^N,\]
hence $2^N = 0$ - but no such $N$ exists.   Also note that $1 \leq N < q-1$, because $\beta^{q-1} = 1$.

\noindent Now suppose $\alpha$ and $\beta$ are prime power equivalent, so that $N$ is a power of $p$: $N = p^s$ for some positive integer $s$.  Then we have
\[ \left( 1 - \beta^k \right)^N = \left( 1 - \beta^k \right)^{p^s} =  1^{p^s} - (\beta^k)^{p^s} =  1 - (\beta^{p^s})^k =  1 - \alpha^k \]
for all $k$, and (\ref{logeqn}) is established.

\noindent For the opposite direction, assume $N = p^sv$, where $s \geq 0$ and $v \not \equiv 0 \mod p$, and suppose that the system (\ref{logeqn}) holds.  Let $G(x) = (1-x)^N +x^N -1 \in \mathbb{F}_q[x]$.  Let $r$ be any nonzero element of $\mathbb{F}_q$,  and set $k = \log_{\beta} r$.  Then using (\ref{logeqn}) we have 
\[    (1-r)^N = \left( 1 - \beta^{\log_{\beta} r} \right)^N = 1 -\alpha^{\log_{\beta}r}= 1 - r^N,\]
hence every element of $\mathbb{F}_q^*$ is a root of $G(x)$.  Now we claim $G(x) \neq 0$.  Expanding the first term $(1-x)^N$ gives us 
\[
    (1-x)^N = \left( 1-x \right)^{p^sv} = (1 - x^{p^s})^v = 1 - x^N + vx^{p^s} + \sum_{t = 2}^{v-1} \binom{v}{k}(-x^{p^s})^{t},\]
hence $G(x) = vx^{p^s} + (\text{\textit{higher-order terms}})$.  In particular, since $v \not \equiv 0 \mod p$, $G(x) \neq 0$.
Therefore $G(x)$ is a nontrivial polynomial with $1 \leq  \deg(G) < q-1$, and so it has less than $q-1$ roots - contradicting the fact that $G(r) = 0$ for all $r \in F_q$.  This contradiction shows that the system (\ref{logeqn}) cannot hold, and the reverse direction is proven.

\end{proof}

\begin{thm}\label{cyclicclas} Let $q=p^n$,  $p$ prime, and let $\alpha$, $\beta$ be two primitive elements in $\mathbb{F}_q$.  
\begin{enumerate}
    \item The quandles $F_{q, \alpha}$ and $F_{q, \beta}$ are isomorphic if and only if $\alpha$ and $\beta$ are prime--power equivalent. 
    \item There are a total of $\displaystyle{\frac{\phi(p^n - 1)}{n}}$ isomorphism classes of cyclic quandles of order $p^n$, where $\phi$ is Euler's totient function. 
\end{enumerate}   
\end{thm}

\begin{proof} Assume $\phi: F_{q, \alpha} \rightarrow F_{q, \beta}$ is a quandle isomorphism, and let $u$ and $v$ be the preimages in $F_{q, \alpha}$ of the generators $x$ and $y$ in $F_{q, \beta}$.  By Lemma \ref{product1} item (2), for any $1 \leq k < q-1$ we have $uv^k = vu^{\log_{\alpha} (1 - \alpha^k)}$.  Applying $\phi$ to this relation gives us 
$xy^k = yx^{\log_{\alpha}(1- \alpha^k)}$.  But since $x,y$ are  the generators in $\in F_{q, \beta}$, we have $xy^k = yx^{\log_{\beta}(1 - \beta^k)}$, so 
\[ yx^{\log_{\alpha}(1- \alpha^k)} = xy^k = yx^{\log_{\beta}(1- \beta^k)}.\]
Therefore by Lemma \ref{product1} item (3), we have $\log_{\alpha}(1-\alpha^k) \equiv \log_{\beta}(1 - \beta^k)$ for all $1 \leq k <q-1$, and then Lemma \ref{loglemma} gives us our result.  

If we now assume that $\alpha$ and $\beta$ are prime--power equivalent, then we have 
\[ \log_{\alpha}(1-\alpha^k) \equiv \log_{\beta}(1 - \beta^k)\]
for all $1 \leq k <q$.  If $F_{q,\alpha}$ is generated by elements $x,y$, and $F_{q,\beta}$ is generated by elements $u, v$, define a map $\phi:F_{q,\alpha} \rightarrow F_{q, \beta}$ by $\phi(x) = u$, $\phi(y) = v$.  Then we can show that $\phi$ is one--to--one, and Lemma \ref{loglemma} now ensures that the defining relations of $F_{q, \alpha}$ are satisfied.

Item (2) is an immediate corollary: given a primitive element $\alpha$ in $\mathbb{F}_q$, any prime power $\alpha^{p^s}$ is also primitive, but $\alpha^{p^n} = \alpha$.  So the corresponding prime--power equivalence class is $\left\{ \alpha, \alpha^p, \ldots, \alpha^{p^{n-1}} \right\}$, with $n$ elements.  
\end{proof}

\begin{ex}
Let $p = 5$.  We construct a field of order $5^3$ as a quotient field of the ring $\mathbb{Z}_5[x]$.  The polynomial $m(x) = x^3 + x + 1$ is irreducible, hence $\mathbb{F}_{5^3} \cong \mathbb{Z}_5[x] / (x^3 + x+ 1)$.  Consider the (left cosets of the) following elements: 
\[ \alpha := 2x^2 + x + 2, \; \; \; \; \beta := 3x^2 + 2x + 1, \; \; \; \; \gamma := 3x^2 + 4x+3\]
These elements are all primitive, since $|\alpha| = |\beta| = |\gamma| = 124$.  Moreover, $\alpha^{5^2} = \beta$, hence $\beta^5 = \alpha$.  So $\alpha$ and $\beta$ are prime--power equivalent primitive elements.  However, $\alpha^{5^s} \neq \gamma$ for any $s$.   So by Theorem \ref{cyclicclas}, we can conclude:
\[ F_{125, \alpha} \cong F_{125, \beta} \not \cong F_{125, \gamma}.\]
There are a total of $\phi(124) = 60$ primitive elements in $\mathbb{F}_{125}$, and the order of each prime--power equivalence class is 3 (e.g., $\left\{ \alpha, \alpha^5, \alpha^{25} \right\}$ is one of these equivalence classes).  So there are a total of 60/3 = 20 equivalence classes, from which we can conclude: up to isomorphism, there are 20 (pairwise nonisomorphic) cyclic quandles of order 125.
\end{ex}

        \section{Representation theory}\label{rep}
   A \emph{representation} of a finite quandle $(X,*)$ on a finite dimensional complex vector space $V$ is a quandle homomorphism $\rho: X \rightarrow GL(V)$, where the automorphism group $GL(V)$ of $V$ is considered as a quandle with conjugation (see \cite{EM}).  In other words, for all $x,y \in X$, we have $\rho(x*y)=\rho(y)\rho(x)\rho(y)^{-1}.$ To simplify the notation, we will denote $\rho(x)$ by just $\rho_x$.  Let $V$ and $W$ be two representations of a quandle $X$.  A map from the representation $V$ to $W$ is a linear map $\phi: V \rightarrow W$ which makes the following diagram commute:
        \[
        \xymatrix{
        V \ar[r]^{\phi} \ar[d]_{\rho^V_x} & W \ar[d]^{\rho^W_x} \\
        V \ar[r]^{\phi} & W.
        }
        \]
        If furthermore $\phi$ is an isomorphism then we say that the two representations are \emph{equivalent}.  If $\phi: V \hookrightarrow W$ is the inclusion of the linear subspace $V$ into $W$, then we say that $V$ is a \emph{subrepresentation} of $W$.\\
        A representation $V$ of a quandle $X$ is called \emph{irreducible} if its only subrepresentations are $\{0\}$ and $V$, and \emph{completely reducible} if it can be written as a direct sum of irreducible subrepresentations.  A representation ia called \emph{indecompasable} if it cannot be written as direct sum of nontrivial subrepresentations.  Thus, clearly every irreducible representation is indecomposable.
        
        \subsection{The regular representation} \label{reg}
        
        For any finite quandle $X$, we denote by $\C X $ the $\C$--vector space of $\C$-valued functions on $X$; equivalently it is the $\C$-vector space generated by basis vectors $\left\{ e_x \right\}_{ x \in X}$,  whose elements are formal sums $f = \sum_{x \in X} a_x e_x$, $a_x \in \C$.  The \textit{regular representation} of $X$ is 
        \[ \lambda: X \rightarrow \text{Conj}(GL(\C X)), \]
        where $\lambda_t(f)(x): = f(R_t^{-1}(x))$. This action of $X$ is equivalent to the right $X$--action on $\C X$ given by the linear extension of $e_x \cdot t = e_{ R_t(x) }=e_{x*t}$.   We note that a subspace $W$ of $\C X$ is a quandle subrepresentation of $
        \C X$ if and only if $W$ is a subrepresentation of $\C X$ as a group representation of $\inn(X)$.
        
        When $|X| > 2$, The regular representation always has two non--trivial subrepresentations, one of which is always irreducible.  This irreducible subrepresentation is the one--dimensional subspace $\C \one$, where $\one$ is the constant function $\one (x) = 1$, for all $x \in X$. Its vector space complement, $(\C \one)^\perp$, is also a subrepresentation, since $(\C \one)^\perp = \left\{ \sum_{x \in X} a_x e_x |\; \sum a_x = 0 \right\}$, and action by any $t \in X$ only permutes the coefficients $a_x$.  So $\C X$ decomposes into a direct sum of subrepresentations 
        \[ \C X = \C \one \oplus (\C \one)^\perp.\]
        For $x, y \in X$ we denote $v_{xy} = e_x - e_y$.  If we enumerate $X = \left\{ x_1, \ldots, x_n \right\}$, then a basis for $(\C \one)^\perp$ is $\left\{ v_{x_1 x_2 }, v_{x_2 x_3}, \ldots, v_{x_{n-1}x_n} \right\}$.
        
        \begin{pro}\label{compdec}
        	The regular representation of a quandle $X$ is completely reducible. 
        	\end{pro}

        \begin{proof} 
        A subspace $W$ of $\C X$ is a quandle subrepresentation of $
        \C X$ if and only if $W$ is a subrepresentation of $\C X$ as a group representation of $\inn(X)$.  Then Maschke's Theorem gives us the result.
        \end{proof}

        \section{The regular representations of dihedral quandles}\label{Rep Dihedral}
        In this section,  we obtain an explicit decomposition of the regular representation of a dihedral quandle.   We denote by $(\Z_n, a)$ the linear quandle $\Z_n$ with operation $x \star y = ax + (1-a)y$, for any $a$ invertible in $\Z_n$.    We denote by $\text{Inn}(\mathbb{Z}_n)$ the group of inner automorphisms of the dihedral quandle $\mathbb{Z}_n$; this is the subgroup of $\Aut(\mathbb{Z}_n)$ generated by the right multiplication operators $R_i$.  
        \begin{thm}
        \cite{EMR} If $n$ is even, $\inn(\mathbb{Z}_n) = D_{n/2}$, the dihedral group of order $n$.  If $n$ is odd, $\inn(\mathbb{Z}_n) = D_{n}$.
        \end{thm}
        
        We describe the (isomorphism classes of) finite--dimensional irreducible group representations of $D_n$ here, to facilitate statement of results below.  The dihedral group $D_n$ has a presentation given by 
        \[ D_n = \left\langle \alpha, \beta \; \; | \; \; |\alpha| = |\beta| = 2, \; |\alpha \beta | = n \right\rangle. \]
        We fix generators $\alpha, \beta$ for this presentation.  The classification of the finite--dimensional irreducible group representations of $D_n$ falls into two cases: $n$ even and $n$ odd.  In either case, let $\omega_r = e^{2 \pi i / r}$.  If $\lambda, \mu \in \C$, we will denote by $\C(\lambda, \mu)$ the one--dimensional group representation of $D_n$ on which $\alpha$ and $\beta$ act by scalar multiplication by $\lambda$ and $\mu$, respectively.  And for $s \in \Z$, we will denote by $W(\omega_r^s)$ the two--dimensional group representation of $D_n$ for which the matrix representations of $\alpha$ and $\beta$ are 
        \[ \alpha \mapsto \begin{bmatrix}
        0 & 1 \\ 
        1 & 0 
        \end{bmatrix}, \; \; \; \; \beta \mapsto \begin{bmatrix}
        0 & \omega_r^s \\
        \omega_r^{-s} & 0 
        \end{bmatrix}.
        \]
        With this notation fixed, we can give the classification of the finite--dimensional irreducible group representations of $D_n$.  All such representations are either 1-- or 2-- dimensional, and are described here:
        
        \begin{center}
        \begin{tabular}{|c|c|c|}
        \hline 
        $D_n$ & $n = 2k$ even & $n = 2k+1$ odd \\
        \hline
        & & \\
        1-dimensional  & $\C(1,1)$, $\C(1,-1)$, & \\ group representations &  
           $\C(-1,1)$, $\C(-1,-1) $& $\C(1,1)$, $\C(-1,-1)$\\
        && \\
        \hline
        &&\\
        2-dimensional &$W(\omega_n^s)$ , & $W(\omega_n^s)$ , \\ group representations &   $1 \leq s \leq k -1 $ &   $1 \leq s \leq k $ \\
        &&\\
        \hline
        \end{tabular}
        \end{center}
        We will denote by $\Gamma(D_n)$ the set of (isomorphism classes of) finite--dimensional irreducible group representations of $D_n$.  For example, $D_8$ has 4 representations of dimension 1 and 3 irreducible representations of dimension 2, and they are 
        \[ \Gamma(D_{8}) = \left\{ \C( \pm 1,\pm 1),  \; \; \C( \pm 1, \mp 1), \; \; W(\omega_8), \; \; W(\omega_8^2), \; \; W(\omega_8^3) \right\}, \]
        while $D_9$ has 2 representations of dimension 1 and 3 irreducible representations of dimension 2, and they are 
        \[ \Gamma(D_{9}) = \left\{ \C( \pm 1,\pm 1),  \; \; W(\omega_9), \; \; W(\omega_9^2), \; \; W(\omega_9^3), \; \; W(\omega_9^4) \right\}. \]
        Now we consider the regular representation $\C \Z_n $ of $\Z_n$ and the subrepresentation $(\C \one)^\perp$, in cases $n$ even or $n$ odd. In either case, $(\C \one)^\perp$ is spanned by $\left\{ v_{ij} \right\}_{1 \leq i, j \leq n }$. 
        
        If $n = 2k$ is even, there are two orbits of $\Z_{n}$ under the quandle action: the even orbit $\left\{ 2, 4, \ldots, 2k \right\}$, and the odd orbit $\left\{ 1, 3, \ldots, 2k-1 \right\}$.  Let 
        \[ \Phi_{n,0} = \text{sp}\left\{ v_{ij} \right\}_{i \neq j \text{ even}}, \; \; \; \Delta_{n,0} = \left\{ v_{24}, v_{46}, \ldots, v_{2k-2 \; 2k} \right\}, \]
         and 
         \[ \Phi_{n,1} = \text{sp}\left\{ v_{ij} \right\}_{i \neq j \text{ odd}}, \; \; \; \Delta_{n,1} = \left\{ v_{13}, v_{35}, \ldots, v_{2k-3 \; 2k-1} \right\}. \]
        We call $\Phi_{n,0}$ and $\Phi_{n,1}$ the even and odd subspaces, respectively, of $(\C \one)^\perp$.   For $i = 0, 1$, $\Delta_{n,i}$ is a basis for $\Phi_{n,i}$.
        \\
        
        If $n$ is odd, we set 
         \[ \Phi_n= \text{sp}\left\{ v_{ij} \right\}_{i \neq j}, \; \; \; \Delta_n = \left\{ v_{12}, v_{23}, \ldots, v_{n-1 \; n} \right\}, \]
         and $\Delta_n$ is a basis for $\Phi_n$.
        
        \begin{thm}\label{dihedral} If $n$ is even, then the decomposition of $\C\mathbb{Z}_n$ into irreducible subrepresentations of $\mathbb{Z}_n$ is, for $n = 4k$, 
          \[ \C\Z_n  \cong  \C( 1, 1)^{\oplus 2} \oplus \C( \pm 1, \mp 1) \bigoplus_{s=1}^{k-1} W(\omega_{2k}^s)^{\oplus 2} \]
        \hspace{.5 in}  And for $n = 4k+2$, \[ \C\Z_n  \cong\C( 1, 1)^{\oplus 2} \bigoplus_{s=1}^{k} W(\omega_{2k+1}^s)^{\oplus 2}. \]
        \\
        \\
        If $n$ is odd, $n = 2r+1$,  then the decomposition of $\C\mathbb{Z}_n$ into irreducible subrepresentations of $\mathbb{Z}_n$ is 
        \[ \C\mathbb{Z}_n  \cong \C (1,1) \bigoplus_{s = 1}^{r} W(\omega_n^{2s}). \]
        \end{thm}
        \begin{proof} For arbitrary $n$, let $\hat{\mathbf{1}} =\sum_{i=1}^n (-1)^{i}e_i \in ( \C \mathbf{1})^\perp$.  Then $\C\mathbb{Z}_n$ decomposes as $ \C\mathbb{Z}_n = \C \mathbf{1} \oplus (\C \mathbf{1})^\perp$, and $\C \mathbf{1} \cong \C \hat{\mathbf{1}} \cong \C(1,1)$.  We proceed in the cases $n$ even or $n$ odd. 
        
        \noindent \textbf{Case 1:} $n= 2r$.  In this case, each $\Phi_{n,i}$ (for $i=0,1)$ is a subrepresentation of $\C\mathbb{Z}_{n}$ of dimension $r-1$, and $(\C \mathbf{1})^\perp$ decomposes as 
        \[ (\C \mathbf{1})^\perp = \C \hat{\mathbf{1}} \oplus \Phi_{n,0} \oplus \Phi_{n,1}. \]
        Set $[ a_1, \ldots, a_{r-1} ]_0 = \displaystyle{\sum_{i=1}^{r-1}{a_i}v_{2i, 2i+2}}$ (the coordinate vector in the basis $\Delta_{n,0}$).  The right multiplication operators $R_1$ and $R_2$ together generate all of $\text{Inn}(\Z_n) \cong D_{r}$, and with respect to the basis $\Delta_{n,0}$, the matrix representations of these operators on the regular representation are 
        \begin{equation}\label{matrixrep}
         \left[ R_1 \right]_{\Delta_{n, 0}} = \begin{bmatrix}
        0 & 0 & \cdots & 0 &-1 \\
        0 & 0 & \cdots & -1 & 0 \\
        \vdots & \vdots & \reflectbox{$\ddots$}  & \vdots & \vdots \\
        0 & -1 & \cdots & 0 & 0 \\
        -1 & 0 & \cdots & 0 & 0
        \end{bmatrix}
        , \; \; \; \;  \left[ R_2 \right]_{\Delta_{n, 0}} = \begin{bmatrix}
        1 & 0 & \cdots & 0 & 0 \\
        1 & 0 & \cdots &  0 & -1 \\
        \vdots & \vdots & \reflectbox{$\ddots$}  & \vdots & \vdots \\
        1 & 0 & \cdots & 0 & 0 \\
        1 & -1 & \cdots & 0 & 0
        \end{bmatrix}.
        \end{equation}
        For example, for $n = 12$, we have 
        \[ R_1\left(  \left[ a_1, a_2, a_3, a_4, a_5 \right]_0 \right) = \left[  -a_5, -a_4, -a_3, -a_2, -a_1 \right]_0,  \]
        \[ R_2\left(  \left[ a_1, a_2, a_3, a_4, a_5 \right]_0 \right) = \left[  a_1, a_1 -a_5, a_1 - a_4, a_1 -a_3, a_1 - a_2 \right]_0. \]
        
        \noindent  Let  $\omega_r = e^{2 \pi i / r}$, and  for $1 \leq s \leq r-1$, let 
        \[ \mathbf{u}_s  :=  \left[ 1-\omega_r^s, 1 - (\omega_r^s)^2, \cdots, 1 - (\omega_r^s)^{r-1} \right]_0 = \sum_{i=1}^{r-1} \left(1-(\omega_r^s)^i\right) v_{2i, 2i+2}, \;  \text{ and } \;  \mathbf{v}_s := R_1(\mathbf{u}_s).\]
        Then $\mathbf{v}_s = \displaystyle{\sum_{i=1}^{r-1} \left((\omega_r^s)^{r-i}-1 \right) v_{2i, 2i+2}}$, and   
        \begin{eqnarray*}
        R_2(\mathbf{u}_s) &=& \sum_{i=1}^{r-1} \left((\omega_r^s)^{r+1-i}-\omega_r^s \right) v_{2i, 2i+2}\\
        &=& \omega_r^s  \sum_{i=1}^{r-1} \left((\omega_r^s)^{r-i}-1 \right) v_{2i, 2i+2}\\
        &=& \omega_r^s \mathbf{v}_s,
        \end{eqnarray*}
        and since all $R_i$ are involutions, we also have $R_2(\mathbf{v}_s) = (\omega_r^{-1})^s \mathbf{u}_s$. Let $W_{s,0}$ be the subspace of $\Phi_{n,0}$ spanned by $\mathbf{u}_s$ and $\mathbf{v}_s$.  The above identities show that each $W_{s,0}$ is a subrepresentation of $\Phi_{n,0}$: if $\mathbf{u}_s$ and $\mathbf{v}_s$ are linearly dependent, then $\mathbf{u}_s$ is an eigenvector for $R_1$ and $R_2$, hence $W_s$ is a 1-dimensional subrepresentation.  And if $\mathbf{u}_s$ and $\mathbf{v}_s$ are linearly independent, the matrix representations of $R_1$ and $R_2$ with respect to the basis $\left\{ \mathbf{u}_s, \mathbf{v}_s \right\}$ are
        \[  \left[ R_1 \right]_{\left\{ \mathbf{u}_s, \mathbf{v}_s \right\} }= \begin{bmatrix}
        0 & 1 \\ 
        1 & 0 
        \end{bmatrix}, \; \; \; \; \left[ R_2 \right]_{\left\{ \mathbf{u}_s, \mathbf{v}_s \right\} } = \begin{bmatrix}
        0 & \omega_r^s \\
        \omega_r^{-s} & 0 
        \end{bmatrix},\]
        which gives us $W_{s,0} \cong W(\omega_r^s)$ (as group representations of $D_r$).  From the identity $(\omega_r^{r-s})^\ell = (\omega_r^s)^{-\ell}$ (for any integers $\ell, s$),  we obtain $\mathbf{u}_s = - \mathbf{v}_{r-s}$, for $1 \leq s \leq r-1$.  Therefore $W_{s,0} = W_{k-s, 0}$, $1 \leq s \leq r-1$.  
        
        Next we will show that the set $\left\{ \mathbf{u}_s \right\}_{1 \leq s \leq r-1}$ is a basis for $\Phi_{n,0}$. Let 
        \[ A = \begin{bmatrix}
          - \mathbf{u}_1 - \\
            - \mathbf{u}_2 - \\
            \vdots \\
               - \mathbf{u}_{k-1} - \\
        \end{bmatrix} = \begin{bmatrix}
        1 - \omega_r & 1 - \omega_r^2  & \cdots & 1 - \omega_r^{r-1} \\
        1 - \omega_r^2 & 1 - (\omega_r^2)^2  & \cdots & 1 - (\omega_r^2)^{r-1} \\
        \vdots & \vdots & & \vdots \\
        1 - \omega_r^{r-1} & 1 - (\omega_r^{r-1})^2 & \cdots & 1 - (\omega_r^{r-1})^{r-1}.
        \end{bmatrix}
        \] 
        Then we can write 
        \[ A = (-1)^s \begin{bmatrix}
        2 & 1  & \cdots & 1 & 1 \\
        1 & 2  & \cdots &  1 & 1 \\
        \vdots &  \vdots & & \vdots & \vdots\\
        1 & 1  & \cdots & 2 & 1 \\
        1 & 1  & \cdots & 1 & 2 
        \end{bmatrix}  \begin{bmatrix}
         \omega &  \omega_r^2  & \cdots &  \omega_r^{r-1} \\
         \omega_r^2 & (\omega_r^2)^2  & \cdots &  (\omega_r^{r-1})^{2} \\
        \vdots & \vdots & & \vdots \\
         \omega_r^{r-1} &  (\omega_r^{2})^{r-1} & \cdots &  (\omega_r^{r-1})^{r-1}.
        \end{bmatrix},\]
        where each matrix is $r-1$ square.  The first matrix in this factorization is easily row-reduced to the identity, while the second matrix is the submatrix obtained from the Vandermonde matrix of the $r$ roots of unity by removing the first row and column of 1's.  Since this Vandermonde matrix is invertible, it follows that $A$ is invertible, hence the set  $\left\{ \mathbf{u}_s \right\}_{1 \leq s \leq r-1}$ is linearly independent.  As a corollary, we obtain  the following: 
        \noindent For $r$  even,  $W_{\ell,0} \cap W_{s,0} = \left\{ 0 \right\}$ for $1 \leq \ell \neq s \leq r/2$,
        and $\displaystyle{\dim(W_{s,0}) = \begin{cases}
        2, & 1 \leq s < r/2\\
        1, & s = r/2
        \end{cases}. }$ 
        For $r$ odd, $W_{\ell,0} \cap W_{s,0} = \left\{ 0 \right\}$ for $1 \leq \ell \neq s \leq (r-1)/2$, and $\dim(W_{s,0}) = 2$ for all~$s$.
        \\
        \\ 
        Hence we obtain a decomposition of $\Phi_{n,0}$ into irreducible subrepresentations: 
        \[ \Phi_{n,0} \cong \begin{cases}
        W_{r/2,0} \displaystyle{\bigoplus_{s=1}^{r/2-1} W(\omega_r^s)}, & r \text{ even}\\
         & \\
        \displaystyle{\bigoplus_{s=1}^{(r-1)/2} W(\omega_r^s)}, & r \text{ odd}\\
        \end{cases}\]
        The decomposition of $\Phi_{n,1}$ is identical - the only change being that $R_1$, $R_2$ are replaced by $R_{n/2}$, $R_1$.  The matrix representations of these transformations, with respect to the basis $\Delta_{n,1}$, are given by 
        
        \begin{equation}\label{matrixrep2}
         \left[ R_{n/2} \right]_{\Delta_{n, 1}} = \begin{bmatrix}
        0 & 0 & \cdots & 0 &-1 \\
        0 & 0 & \cdots & -1 & 0 \\
        \vdots & \vdots & \reflectbox{$\ddots$}  & \vdots & \vdots \\
        0 & -1 & \cdots & 0 & 0 \\
        -1 & 0 & \cdots & 0 & 0
        \end{bmatrix}
        , \; \; \; \;  \left[ R_1 \right]_{\Delta_{n, 1}} = \begin{bmatrix}
        1 & 0 & \cdots & 0 & 0 \\
        1 & 0 & \cdots &  0 & -1 \\
        \vdots & \vdots & \reflectbox{$\ddots$}  & \vdots & \vdots \\
        1 & 0 & \cdots & 0 & 0 \\
        1 & -1 & \cdots & 0 & 0
        \end{bmatrix}.
        \end{equation}
        The remainder of the argument is the same, \textit{mutatis mutandis}, and shows 
        
        \[ \Phi_{n,1} \cong \begin{cases}
        W_{r/2,1} \displaystyle{\bigoplus_{s=1}^{r/2-1} W(\omega_r^s)}, & r \text{ even}\\
         & \\
        \displaystyle{\bigoplus_{s=1}^{(r-1)/2} W(\omega_r^s)}, & r \text{ odd}\\
        \end{cases}\]

        For $r$ even, the one-dimensional subrepresentations of $\C\mathbb{Z}_{2k}$ are:
        \[ \C \mathbf{1}, \; \; \; \; \C\left( \sum_{i=0}^n (-1)^{i}e_i \right), \; \; \; W_{r/2,0}, \; \; \; \text{ and } \; \; \; W_{r/2,1}, \]
        and the two-dimensional irreducible representations are $W_{s,0}$ and $W_{s,1}$,  $1 \leq s \leq r/2-1$.
        
        For $r$ odd,  the one-dimensional subrepresentations of $\C\mathbb{Z}_{2r}$ are:
        \[ \C \mathbf{1} \; \;  \text{ and }\; \; \C\left( \sum_{r=0}^n (-1)^{i}e_i \right) ,\]
        and the two-dimensional irreducible representations are $W_{s,0}$ and $W_{s,1}$,  $1 \leq s \leq (r~-~1)~/~2$.
        
        Next assume $n$ is odd: $n = 2r+1$.  In this case, $\inn(\Z_n) \cong D_n$.  With respect to the basis $\Delta_n = \left\{ v_{12}, v_{23}, \ldots, v_{n-1, n} \right\}$, the elements $R_1$, $R_2 \in \inn(Z_n)$ have matrix representation 
        \begin{equation*}
         \left[ R_1 \right]_{\Delta_n} = \begin{bmatrix}
        0  & \cdots & 0 & -1 & 1 \\
        0  & \cdots  & -1 & 0  & 1 \\
        \vdots  & \reflectbox{$\ddots$}  & \vdots & \vdots & \vdots \\
        -1 & 0 & \cdots & 0 & 1 \\
        0 & 0 & \cdots & 0 & 1
        \end{bmatrix}
        , \; \; \; \;  \left[ R_2 \right]_{\Delta_n} = \begin{bmatrix}
        1 & 0 & \cdots & 0 & 0 \\
        1 & 0 & \cdots &  0 & -1 \\
        \vdots & \vdots & \reflectbox{$\ddots$}  & \vdots & \vdots \\
        1 & 0 & \cdots & 0 & 0 \\
        1 & -1 & \cdots & 0 & 0
        \end{bmatrix}.
        \end{equation*}
         Let  $\omega = e^{2 \pi i / n}$, and  for $1 \leq s \leq n-1$, let 
        \[ \mathbf{u}_s  =  \left[ 1-\omega^s, 1 - (\omega^s)^2, \cdots, 1 - (\omega^s)^{n-1} \right], \; \; \; \text{ and } \; \; \mathbf{v}_s = R_1(\mathbf{u}_s).\]
        We can then show that 
        \[ \mathbf{v}_s = \omega^{-s}\left( \omega^{-s} - 1, (\omega^{-s})^2 - 1, \ldots, (\omega^{-s})^{n-1} - 1\right).\]
        Let $W_{s}$ be the subspace of $\Phi_{n}$ spanned by $\mathbf{u}_s$ and $\mathbf{v}_s$.   We can show that $\mathbf{u}_s = \mathbf{v}_{n-s}$, hence $W_s = W_{n-s}$.  Furthermore, we have 
        \begin{eqnarray*}
        R_2(\mathbf{u}_s) &=& R_2( 1 - \omega^s, 1 - (\omega^s)^2, \ldots, 1 - (\omega^s)^{n-1})\\
        &=& ( 1 - \omega^s, (\omega^s)^{n-1} - \omega^s, (\omega^s)^{n-2} - \omega^s, \ldots, (\omega^s)^{2} - \omega^s) \\
         &=& \omega^{2s} \omega^{-s} ( \omega^{-s} -1, (\omega^{-s})^2 -1, \ldots, (\omega^{-s})^{n-1} - 1)\\
         &=& \omega^{2s} \mathbf{v}_s.
        \end{eqnarray*}
        Therefore, if $\mathbf{u}_s$ and $\mathbf{v}_s$ are linearly independent, the matrix representations of $R_1$ and $R_2$ with respect to the basis $\left\{ \mathbf{u}_s, \mathbf{v}_s \right\}$ of $W_s$ are
        \[  \left[ R_1 \right]_{\left\{ \mathbf{u}_s, \mathbf{v}_s \right\} }= \begin{bmatrix}
        0 & 1 \\ 
        1 & 0 
        \end{bmatrix}, \; \; \; \; \left[ R_2 \right]_{\left\{ \mathbf{u}_s, \mathbf{v}_s \right\} } = \begin{bmatrix}
        0 & \omega^{2s} \\
        \omega^{-2s} & 0 
        \end{bmatrix},\]
        which gives us $W_{s} \cong W(\omega^{2s})$ (as group representations of $D_n$).  We can show linear independence of $\mathbf{u}_s$ and $\mathbf{v}_s$ for $1 \leq s \leq r$, which gives us the desired decomposition into irreducible representations 
        \[ \C\mathbb{Z}_n  \cong \C (1,1) \bigoplus_{s = 1}^{r} W(\omega_n^{2s}). \]
        \end{proof}
        \begin{ex} We list the irreducible subrepresentations of $\C \Z_{10}$, $\C \Z_{11}$, and $\C \Z_{12}$:
        \begin{center}
        \begin{tabular}{clllcl}
        subrep of $\Z_{10}$ & & irrep of $D_5$ & generated by   & dimension\\[4pt]
        \hline\\[-4pt]
         $\C \mathbf{1}$ & $\cong$ &  $\C(1,1) $ & $\hspace{.5in} \mathbf{1}$ & 1\\[4pt]
        $\C \hat{\mathbf{1}} $ & $ \cong $ &$\C(1,1)$ &  \hspace{.45in} $\hat{\mathbf{1}}$ & 1 \\[4pt]
        $W_{1,0}$ & $ \cong $ & $W(\omega_5) $ & $\sum_{i=1}^4 \left( 1 - \omega_5^i
        \right)v_{2i, 2i+2}$ \hspace{.37in} &  2 \\[4pt]
        $W_{2,0}$ & $ \cong $ & $W(\omega_5^2 ) $ & $\sum_{i=1}^4 \left( 1 - \omega_5^{2i}\right)v_{2i, 2i+2}$ & 2 \\[4pt]
        $W_{1,1}$ & $ \cong $ & $W(\omega_5) $& $\sum_{i=1}^4 \left( 1 - \omega_5^{i}\right)v_{2i-1, 2i+1}$   & 2\\[4pt]
        $W_{2,1}$ & $ \cong $ & $W(\omega_5^2 ) $& $\sum_{i=1}^4 \left( 1 - \omega_5^{2i}\right)v_{2i-1, 2i+1}$ & 2 \\[4pt]
        \end{tabular}
        \end{center}
        
        \begin{center}
        \begin{tabular}{clllcl}
        subrep of $\Z_{11}$ & & irrep of $D_{11}$ & generated by   & dimension\\[4pt]
        \hline\\[-4pt]
        $\C \mathbf{1}$ & $\cong$ &  $\C(1,1) $ &  \hspace{.5in} $ \mathbf{1}$ & 1\\[4pt]
        $W_{1}$ & $ \cong $ & $W(\omega_{11}^2) $&  $ \sum_{i=1}^8 \left( 1 - \omega_{11}^{2i}\right)e_i$ \hspace{.58in} & 2 \\[4pt]
        $W_{2}$ & $ \cong $ & $W(\omega_{11}^4) $&  $\sum_{i=1}^8 \left( 1 - \omega_{11}^{4i}\right)e_i$& 2 \\[4pt]
        $W_{3}$ & $ \cong $ & $W(\omega_{11}^6) $&  $\sum_{i=1}^8 \left( 1 - \omega_{11}^{6i}\right)e_i$& 2 \\[4pt]
        $W_{4}$ & $ \cong $ & $W(\omega_{11}^8) $&  $\sum_{i=1}^8 \left( 1 - \omega_{11}^{8i}\right)e_i$& 2 \\[4pt]
        $W_{5}$ & $ \cong $ & $W(\omega_{11}^{10}) $&  $\sum_{i=1}^8 \left( 1 - \omega_{11}^{10i}\right)e_i$& 2 \\[4pt]
        \end{tabular}
        \end{center}
        
        \begin{center}
        \begin{tabular}{clllcl}
        subrep of $\Z_{12}$ & & irrep of $D_6$ & generated by   & dimension\\[4pt]
        \hline\\[-4pt]
        $\C \mathbf{1}$ & $\cong$ &  $\C(1,1) $ & $\hspace{.5in} \mathbf{1}$ & 1\\[4pt]
        $\C \hat{\mathbf{1}} $ & $ \cong $ &$\C(1,1)$ &  \hspace{.45in} $\hat{\mathbf{1}}$ & 1 \\[4pt]
        $W_{1,0}$ & $ \cong $ & $W(\omega_6) $ & $\sum_{i=1}^5 \left( 1 - \omega_6^i
        \right)v_{2i, 2i+2}$ &  2 \\[4pt]
        $W_{2,0}$ & $ \cong $ & $W(\omega_6^2 ) $ & $\sum_{i=1}^5 \left( 1 - \omega_6^{2i}\right)v_{2i, 2i+2}$ & 2 \\[4pt]
        $W_{3,0}$ & $ \cong $ & $\C(-1,1)$ & $\sum_{i=1}^5 \left( 1 - 
        (-1)^{i}\right)v_{2i, 2i+2}$   & 1 \\[4pt]
        $W_{1,1}$ & $ \cong $ & $W(\omega_8) $& $\sum_{i=1}^5 \left( 1 - \omega_6^i\right)v_{2i-1, 2i+1}$   & 2\\[4pt]
        $W_{2,1}$ & $ \cong $ & $W(\omega_6^2 ) $& $\sum_{i=1}^5 \left( 1 - \omega_6^{2i}\right)v_{2i-1, 2i+1}$ & 2 \\[4pt]
        $W_{3,1}$ & $ \cong $ & $\C(1,-1) $& $\sum_{i=1}^5 \left( 1 - (-1)^i \right)v_{2i-1, 2i+1}$  & 1\\[4pt]
        \end{tabular}
        \end{center}

        \end{ex}

\section{Reducible and irreducible representations of cyclic quandles}\label{sec:irreps}

We consider in this section representations of cyclic quandles of arbitrary cardinality. We start by studying $2$-dimensional  representations in detail. 

\begin{thm}\label{thm:cyclic_irrep}
    Let $\phi : X \longrightarrow {\rm Conj}({\rm Aut}\C^2)$ be a non-constant representation of a cyclic quandle $X$ of order $n\geq 3$. Then $\phi$ is irreducible, or $A^{q-1}$ and $B^{q-1}$ are multiples of the identity. 
\end{thm}

   \begin{proof}

    Let us assume, by way of contradiction, that $\phi$ is reducible, and that it is non-constant such that $A^{q-1}$ is not a multiple of the identity. Let $A$ and $B$ be different $2\times 2$ matrices with $\phi(x) = A$ and $\phi(y) = B$, and let $A$ be in Jordan canonical form. Since the quandle $X$ is connected, $B = CAC^{-1}$ for some invertible matrix $C$. We distinguish two cases, depending on the Jordan canonical form of $A$ consisting of a single block, or two distinct blocks (i.e. $A$ diagonalizable). Suppose first that $A = \begin{bmatrix} \lambda & 1\\ 0 & \lambda \end{bmatrix}$. Further, we set $C = \begin{bmatrix} a & b\\ c & d \end{bmatrix}$, where $det(C) \neq 0$. Since $\phi$ is reducible, $A$ and $B$ have common eigenvectors. Using Shemesh's Criterion (\cite{Shemesh}), this is equivalent to saying that ${\rm ker} [A,B] \neq 0$. Using $B=CAC^{-1}$, for the commutator of $A$ and $B$ we find 
    $$
    [A,B] = det(C)\begin{bmatrix} c(a\lambda -c + d\lambda) & -a\lambda + ac - cd + d^2\lambda\\ 0 & -c(a\lambda -c + d\lambda) \end{bmatrix}.
    $$
    Since ${\rm ker}([A,B]) \neq 0$, we have that $c(a\lambda -c + d\lambda) = 0$. From this condition, we obtain two subcases, namely when $c=0$ and when $-a\lambda + c - d\lambda = 0$. Let us assume that $c=0$ first. We argue that it must also hold $-a\lambda + ac - cd + d^2\lambda = 0$. In fact, suppose this is not the case. Then we can rewrite (up to some normalizing constant) $C=\begin{bmatrix}
        t&\mu\\
        0 & 1
    \end{bmatrix}$ for some $t$ and $\mu$. Then we obtain $B = \begin{bmatrix}
        \lambda & t\\
        0 & \lambda
    \end{bmatrix}$, which implies $[A,B] = 0$, against the fact that we assumed that one of the entries of $[A,B]$ was nonzero. 
    In the case $c=0$, then, from $-a\lambda + ac - cd + d^2\lambda = 0$ we obtain $d = \pm a$. Therefore, up to rescaling $C$ by a constant (which does not affect the conjugation of $A$ by $C$) we can assume $C = \begin{bmatrix} 1 & \mu\\ 0 & 1 \end{bmatrix}$. Therefore it follows that $CAC^{-1} = A$, which implies $A=B$ against our initial assumptions. Let us now consider $-a\lambda + c - d\lambda = 0$. In this case we have $c= a\lambda + d\lambda$, from which we have $C = \begin{bmatrix} a & b\\ (a+d)\lambda & d \end{bmatrix}$. Then, we have 
    $$
    CAC^{-1} = \begin{bmatrix}
    \lambda[ad - (a+d)(a+b\lambda)] & a^2\\
    -\lambda^2(a+d)^2 & \lambda[a(a+2d) - b\lambda(a+d)]
    \end{bmatrix}.
    $$
    From the relation $A^{q-1}BA^{-q+1} = B$ we obtain, using the fact that $B = CAC^{-1}$, the equality $A^{q-1}CAC^{-1}A^{-q+1} = CAC^{-1}$. 
    Moreover, using $A = \begin{bmatrix}
    1& 1/\lambda\\
    0& 1
    \end{bmatrix}$
     we find that the LHS of this equation gives (through an induction argument in $q$)
     \begin{eqnarray*}
       \lefteqn{A^{q-1}CAC^{-1}A^{-q+1}}  \\
       &=& \tiny{ \begin{bmatrix}
        ad-(q-1)(a+d)^2-(a+d)(a+b\lambda) & a^2/\lambda + \frac{(q-1)[(a(a+d)-b\lambda(a+b)]}{\lambda}\\
        -\lambda(a+d)^2 & a(a+(q-1)d)-b\lambda(a+d) + (q-1)(a+d)^2
     \end{bmatrix}},  
     \end{eqnarray*}
     while for the RHS one has
     $$
     CAC^{-1} = \small{\begin{bmatrix}
     ad-(a+d)(a+b\lambda) & a^2/\lambda \\
     -\lambda(a+d)^2 & a(a+(q-1)d)-b\lambda(a+d) + (q-1)(a+d)^2
     \end{bmatrix}}.
     $$
     Comparing the top-left entries of the two matrices we obtain $(a+d)^2 = 0$, which in turn implies $a=-d$. But since we started from the assumption that $c = \lambda(a+d)$, we find that $c=0$, and we can apply the same argument as for the previous case to show that $A=B$, against the assumption that the representation $\phi$ is not constant. This contradiction completes the case where $A$ consists of a single Jordan block (i.e. it is not diagonalizable). 

     Let us now suppose that $A = \begin{bmatrix}
         \lambda & 0\\
         0 & \mu
     \end{bmatrix}$, where clearly $\mu\neq\lambda$, or we would immediately have a contradiction. As before, using $B = CAC^{-1}$ for some generic  invertible $C$, we obtain for the commutator
     $$
     [A,B] = \begin{bmatrix}
         0 & b(a\mu(\lambda-\mu) + d\lambda^2 -d\mu\lambda)\\
         c(-a\lambda\mu + a\mu^2 + d\lambda(\lambda-\mu)) & 0
     \end{bmatrix}.
     $$
     Once again, using the hypothesis of reducibility of $\phi$ and Shemesh's Criterion, it follows that ${\rm ker}[A,B]\neq 0$. As a consequence, this forces $b(a\mu(\lambda-\mu) + d\lambda^2 -d\mu\lambda) = 0$, or  $c(-a\lambda\mu + a\mu^2 + d\lambda(\lambda-\mu)) = 0$. If either $b=0$ or $c=0$, then we can proceed similarly to the previoius case to see that we obtain a contradiction. We can then assume that $a\mu(\lambda-\mu) + d\lambda^2 -d\mu\lambda=0$ or $-a\lambda\mu + a\mu^2 + d\lambda(\lambda-\mu)=0$. Either way we obtain $d\lambda = a\mu$.  Therefore, we can write $C = \begin{bmatrix}
         a&b\\
         c&a\mu/\lambda
     \end{bmatrix}$.
     From the relation $CAC^{-1} = A^{q-1}CAC^{-1}A^{-q+1}$, equating the terms corresponding to entries $1,2$ in the matrices, we find that $\frac{ab\lambda(\lambda-\mu)(\lambda^{q-1}-\mu^{q-1})}{\mu^{q-1}(a^2\mu-bc\lambda)} = 0$. This implies (since we had $b\neq0$ and $\lambda\neq \mu$) that $a=0$ (and hence $d=0$ as well), or $\lambda^{q-1}-\mu^{q-1}=0$. If $a=0$, then the matrix $C$ simplifies to an anti-diagonal matrix, and a direct computation using the relations again shows that this is not possible. It must therefore hold that $\lambda^{q-1}-\mu^{q-1}=0$, from which it follows that $A^{q-1}$ is a multiple of the identity.
     The representation $\phi$ must therefore be irreducible. 
   \end{proof}

We now consider $2$-dimensional irreducible representations of cyclic quandles of arbitrary order $n\geq 3$.

\begin{pro}\label{pro:no_irrep}
    Let $X$ be a cyclic quandle of order $n\geq 3$, and let $\phi : X \longrightarrow {\rm Conj}({\rm Aut}\C^2)$ be a representation of $X$ that does not map the generators to mutliples of $(q-1)$-roots of the identity. Then $\phi$ is reducible. 
\end{pro}

\begin{proof}
    Let us assume that $\phi$ is irreducible. We denote by $\phi(x) = A$ and $\phi(y)=B$ the images of the generators of $X$. Let $e$ denote an eigenvector of $A$ with eigenvalue $\lambda$. Then, from one of the defining relations, we have that $A^{q-1}BA^{-q+1}e=Be$, which implies
    $$
    \lambda^{-q+1}A^{q-1}Be = Be, 
    $$
    which means that $Be$ is an eigenvector of $A^{q-1}$ with eigenvalue $\lambda^{q-1}$. Also, $e$ is an eigenvector of $A^{q-1}$ with eigenvalue $\lambda^{q-1}$ because $e$ is an eigenvector of $A$ with eigenvalue $\lambda$. Since $\phi$ is irreducible, then $e$ and $Be$ are not proportional, and therefore $\{e,Be\}$ is an eigenbasis for $A^{q-1}$ with same eigenvalue. In other words, it follows that $A^{q-1}$ is the multiplication by the constant $\lambda^{q-1}$, which is not possible by hypothesis. Therefore $\phi$ needs to be reducible. 
\end{proof}

\begin{cor}
    Let $X$ be a cyclic quandle of order $n\geq 3$ and let $\phi$ be a representation of $X$. Then $\phi$ is a constant map $X\longrightarrow {\rm Conj}({\rm Aut}\C^2)$, or $\phi$ maps the generators (up to a constant) to $(q-1)^{\rm th}$-roots of the identity.
\end{cor}    
     \begin{proof}

        If $\phi$ is not constant and it does not map the generators to roots of the identity, then it is irreducible by Theorem~\ref{thm:cyclic_irrep}. But this is not possible by Proposition~\ref{pro:no_irrep}. 
    
     \end{proof}
For any cyclic quandle $F(q, \alpha)$, presently it is unknown by the authors whether non--constant representations $\phi: F(q, \alpha) \rightarrow \text{Conj}(\Aut(V))$ exist.  However, there is a condition on the minimal polynomial of $\phi(0)^{q-1}$ which guarantees this representation to be constant. We describe this result in the remainder of this section.

We will denote by $J(\lambda, s)$ the $s\times s$ Jordan block matrix 
\[ J(\lambda, s) =  \begin{bmatrix} 
                   \lambda & 1  & \cdots  & 0 \\
                     0 & \lambda  & \ddots & 0 \\
                     \vdots &  \ddots & \ddots & 1\\
                     0 & 0 &   \cdots & \lambda 
                     \end{bmatrix}. \]
                     
Let $A$ be an $n \times n$ matrix in Jordan Normal form, with $r$ Jordan blocks and corresponding eigenvalues $\lambda_1, \ldots, \lambda_r$.  We will write $A = J(\lambda_1, s_1) \oplus J(\lambda_2, s_2) \oplus \cdots \oplus J(\lambda_r, s_r)$ (where the $i^{th}$ Jordan block has size $s_i$).  The $k^{th}$ powers of these values $\lambda_1^k, \ldots, \lambda_r^k$ are all pairwise distinct if and only if the minimal polynomial of $A^k$ has maximal degree $n$.  For this reason, we will say that $A$ is $k^{th}$\textit{--power maximal} if this condition on its Jordan block values ($\lambda_i^k \neq \lambda_j^k$ for $i \neq j$) is satisfied. For example, consider the following matrices: 
\[ A = \begin{bmatrix}
    1 & 1 & 0  & 0 \\
    0 & 1 & 0  & 0 \\
    0 & 0 & i  & 1 \\
    0 & 0 & 0  & i
\end{bmatrix} = J(1, 2) \oplus J(i, 2), \; \; \;  B = \begin{bmatrix}
    1 & 1 & 0  & 0 \\
    0 & 1 & 0  & 0 \\
    0 & 0 & 2i  & 1 \\
    0 & 0 & 0  & 2i
\end{bmatrix} = J(1, 2) \oplus J(2i, 2).\]
Then $A$ is $k^{th}$--power maximal for all $k$ with $k \not \equiv 0 \mod 4$, while $B$ is $k^{th}$--power maximal for all $k$.

If $X = \mathbb{F}(q, \alpha)$ is a quandle of cyclic type, we will say that a representation $\phi: X \rightarrow \text{Conj}(\Aut(V))$ is \textit{maximal} if $\phi(0)$ is $(q-1)^{th}$--power maximal.

If $M$ is any $n\times n$ matrix, the constant map $\phi_M: \mathbb{F}(q, \alpha) \rightarrow \Aut(\mathbb{C}^n)$ given by $\phi_M(0) = \phi_M(1) = M$ is an $n$--dimensional representation of $\mathbb{F}(q, \alpha)$; for brevity we denote this representation by $U_M$ (isomorphic to $\mathbb{C}^n$ as a vector space).  In particular, $U_{J(\lambda, k)}$ yields an indecomposable - yet reducible (for $k>1$) - representation  of $\mathbb{F}(q, \alpha)$, and if a matrix $A$ is in Jordan normal form $A = J(\lambda_1, s_1) \oplus J(\lambda_2, s_2) \oplus \cdots \oplus J(\lambda_r, s_r)$, then we have an isomorphism of quandle representations
\[ U_A \cong U_{J(\lambda_1, s_1)} \oplus U_{J(\lambda_2, s_2)} \oplus \cdots \oplus U_{J(\lambda_r, s_r)}.\]

 \begin{thm}\label{thm:higher_dim_cyclic}
Let $X = \mathbb{F}(q, \alpha)$ be a quandle of cyclic type, and let $\psi: X \rightarrow \mathrm{Conj}(\Aut(V))$ be a maximal quandle representation of $X$ of dimension $> 1$.  Then $\psi$ is a constant map.  If $\rho$ is the characteristic polynomial of $\psi(0)$ and $n_z$ the multiplicity of $z$ as a root of $\rho$, then we have a quandle isomorphism
\[ V \cong \bigoplus_{\rho(z) = 0} U_{J(z, n_z)}. \]
In particular, there are no non--trivial irreducible maximal representations of $X$. 
\end{thm}
\begin{proof}
 Let $J = \psi(0)$, and $M = \psi(1)$.  Assume without loss of generality that $J$ is in Jordan canonical form.  Then $J$, $M$ is a pair of matrices that satisfies 
\[ J^{q-1}MJ^{1-q} = M, \; \; M^{q-1} J M^{1-q} = J, \]
\[ M^k J M^{-k} = J^{\log_{\alpha}(1 - \alpha^k)} M J^{-\log_{\alpha}(1 - \alpha^k)}, \; \; \; \;  1 \leq k \leq q-2.\]
It follows that $J$ and $M$ are similar, hence they have the same Jordan canonical form.  
By Corollary \ref{cylic_corollary}, we must have $J = M$.  Therefore $\psi: X \rightarrow \text{Conj}(\text{Aut}(V)) $ is the constant map $\psi: X \mapsto J$, hence $V \
\cong U_{J}\cong \bigoplus_{\rho(z) = 0} U_{J(z, n_z)}$.  
\end{proof}

        \section{Examples of non-decomposability}\label{sec:non_decomposability}

        We give first an example of a quandle homomorphism $q:S_3 \rightarrow S_3$, where $S_3=\langle r, \theta \mid r^3=1=\theta^2, \theta r \theta=r^2 \rangle $ is considered with conjugation operation, that is not a group homomorphism. We want to define thus, a map satisfying:
        \begin{equation*}\label{con}
        q(yxy^{-1})=q(y) q(x) q(y)^{-1},
        \end{equation*}  for all $x, y \in S_3$, that is not a group homomorphism. Choose any nontrivial element $R \neq 1$.  Then define a map $q$ by mapping $1, \theta, \theta r, \theta  r^2$ to $1$ and mapping $r, r^2$ to $R$.  It is straightforward to see that $q$ is a quandle homomorphism.  However, since $q(\theta \; r)=1$ and $q(\theta)q(r)=R$ it follows that $q$ is \textit{not} a group homomorphism.
        
        A quandle representation induces a quandle  representation on the inner group of the quandle, but the latter does not necessarily defines a group representation. All this suggests that representing a quandle might be much different than representing a group, even in the most common case in which the quandle we are considering is a conjugation quandle. In fact this suggestion turns out to be correct. We show by an explicit example that, by contrast to the case of groups, complementary invariant subspaces may not exist for some representations of quandles. As a consequence, Maschke's Theorem does not hold for quandles.
        
        \begin{ex}
        	Fix a positive integer $n$ greater than or equal to $2$ and consider the dihedral quandle $\mathbb{Z}_{2n}= \{0, 2, \cdots, 2n\} \sqcup \{1, 3, \cdots , 2n-1\}$ as a union of its two orbits. Consider the map $\mathbb{Z}_{2n} \rightarrow GL(2,\C)$ sending the orbit $\{0, 2, \cdots, 2n\}$ to the identity matrix and the orbit $\{1, 3, \cdots , 2n-1\}$ to the matrix
        $	\begin{bmatrix}
        		1  & 1 \\
        		0  & 1 
        	\end{bmatrix}.$  It is clear that $W=\C e_1$ is invariant, where $e_1 = \begin{bmatrix}
        		1  \\
        		0 
        	\end{bmatrix}$.  A simple check shows that $W$ does not have a complementary invariant subspace inside $\C^2$.  More generally, we can define a quandle representation of $\mathbb{Z}_{2n}$ into $GL(m,\C)$ by mapping the orbit $\{0, 2, \cdots, 2n\}$ to the identity matrix and the orbit $\{1, 3, \cdots , 2n-1\}$ to any invertible matrix $B$.  Assume that 
        	\[
        1+	\sum _{i=1}^k \mu_G(\lambda_i) = 	\sum _{j=1}^k \mu_A(\lambda_i), 
        	\] 
        	where $\lambda_i$ are the eigenvalues of $B$ and $\mu_G(\lambda_i)$ and $\mu_A(\lambda_i)$ are respectively the geometric and algebraic multiplicities of $\lambda_i$.  In this case the representation is not completely reducible. In general, the same procedure will work for non-connected quandles. 
        \end{ex}
        
        The representation theory of quandles, therefore, appears to be drastically different from the representation theory of groups. The easiest quandle we can think of, the trivial quandle on one element admits non completely reducible representations: Consider the representation $ \{1\} \longrightarrow GL_2(\mathbb{R})$ defined by $ 1 \mapsto \begin{bmatrix}
        1  & 1 \\
        0  & 1 
        \end{bmatrix}$. As above, it is easy to verify that this representation is not completely reducible. The same procedure shows that the trivial quandle of any cardinality admits non completely reducible representations and more generally, any quandle admits a non completely reducible representation.
        
        \begin{ex}
        	Let $X = \sqcup_i X_i$ be a quandle partitioned into its orbits. Then map the whole orbit $X_1$ to the (usual) matrix $ \begin{bmatrix}
        		1  & 1 \\
        		0  & 1 
        	\end{bmatrix}$, and any other orbit to the identity matrix. This representation is not completely reducible. 
        	\end{ex}

        %
        
         \section{Appendix}

         In this Appendix we state and prove some technical results.  We also consider some explicit constructions of representations. We show that reducible representations of quandles need not be completely reducible, and we also give some computations for the dihedral quandle of order $6$ which exhibit its decomposition into irreducible subrepresentations.

  \begin{lem}\label{nonlinearsystem}
 For any prime power $q$ and $\alpha$ a primitive root in $\mathbb{F}_q$, the system of equations 
        \[ 1 -  x^k = x^{\log_{\alpha}(1 - \alpha^k)}, \; \; 1 \leq k \leq q-2\]
        has no solutions.
   \end{lem}     
   \begin{proof}
   
We will prove a more general statement: \textit{Let $M > 1$,  and suppose $\phi$ is an order-2 permutation of $\left\{ 1, 2,  \ldots, M \right\}$ with a fixed point $N$.  Let $S$ be the set of equations 
\[ S = \left\{ x^k + x^{\phi(k)} - 1 = 0, \; \; \; 1 \leq k \leq M \right\}.\]
Then $S$ has no simultaneous solutions (in $\mathbb{C}$).}

The proof of this statement follows.  Since $N$ is a fixed point, $S$ contains the equation $2x^N -1 =0$.  Furthermore, if we add all equations in $S$, we get $\sum_{i=1}^N x^i = M/2$.  Adding 1 to both sides of this equation gives us 
\[   \frac{1 - x^{M+1}}{1-x} = \sum_{i=0}^M x^i = \frac{M+2}{2}. \]
So any solution of $S$ must also be a solution of the system 
\begin{equation}\label{nonlinear1}
2x^N - 1 = 0,
\end{equation}
\begin{equation}\label{nonlinear2}
1- x^{M+1} = \frac{M+2}{2}(1-x).
\end{equation}
The polynomial $f(x)= x^N - 2$ is irreducible over $\mathbb{Q}$ by Eisenstein's criterion, hence its reciprocal $\tilde{f}(x):= 2x^N-1$ appearing in Equation \ref{nonlinear1} is irreducible over $\mathbb{Q}$ as well.  

Using the division algorithm, we write 
\[ M+1 = bN+r, \; \; \; 0 \leq r < N.\]
Now suppose $x_0$ is a solution for $S$. Then it must also satisfy Equations (\ref{nonlinear1}) and (\ref{nonlinear2}).  Therefore $x_0^N = 1/2$, $x_0^{M+1} = x_0^{bN + r} = \left( x_0 ^N \right)^b x_0^r = (1/2)^b x_0^r$, and so
\[ 1 - \left( \frac{1}{2} \right)^b x_0^r = 1 -x_0^{M+1} = \frac{M+2}{2}(1- x_0).\]
Therefore $x_0$ is a root of a polynomial in $\mathbb{Q}[x]$ of degree $r$, hence the minimal polynomial of $x_0$ (over $\mathbb{Q}$) must have degree $ \leq r$.  On the other hand, since $\tilde{f}(x) = 2x^N -1$ is irreducible over $\mathbb{Q}$ and $2x_0^N - 1 = 0$,  the minimal polynomial of $x_0$ has degree $N > r$.   This contradiction shows that $S$ can have no solutions, and the statement at the beginning of the proof is demonstrated. 

To conclude the proof of the Lemma, note that the map $k \mapsto \log_{\alpha}(1-\alpha^k)$ is an order-2 bijection of the set $\left\{ 1, \ldots, q-2 \right\}$, with one fixed point (this fixed point is $N = -\log(2) \text{ mod } (q-1)$).  Now it follows that the set of equations 
    \[ \left\{  x^k + x^{\log_{\alpha}(1 - \alpha^k)} -1 = 0, \; \; \; 1 \leq k \leq q- 2 \right\} \]
  has no solutions.

\end{proof}

In general, if a square matrix $A$ is upper triangular, a $k^{th}$ root of $A$ ($X$ such that $X^k = A$) is not necessarily upper triangular.  But a sufficient condition for this to hold is given here:
\begin{lem}\label{upper_triangular}
Let $A$ be an $n \times n$ square matrix and $k$ a positive integer such that $A$ is $k^{th}$--power maximal.  If $A^k$ is upper triangular, then $A$ is also upper triangular. 
\end{lem}
\begin{proof}
Suppose $A$ satisfies the given conditions, and let the diagonal entries of $A^k$ be $a_1, \ldots, a_n$.  For $1 \leq j \leq n$, let 
\[ T_j = \prod_{i=1}^j \left( A^k -a_i I \right), \]
and let $W_j = \text{span}(e_1, \ldots, e_j)$ (where $e_i$ is the $i^{th}$ standard basis vector). 

Upper triangularity of $A^k$ means that all subspaces $W_j$ are $A^k$--stable.  Also note that any solution $X$ of $X^k = A^k$ will commute with $A^k$, and therefore the kernel of any polynomial in $A^k$ will be $X$--stable. If we can show that each $W_j$ is $A$--stable, it will follow that $A$ is upper triangular.  We will do this by showing that each $W_j$ is the kernel of the polynomial $\displaystyle{T_j(\lambda) = \prod_{i=1}^j (\lambda - a_i)}$, evaluated at $\lambda = A^k$. 

An inductive argument shows that upper triangularity of $A^k$ guarantees containment $W_j \subseteq \ker(T_j)$, for all $j$.  To show equality, we will argue that all subspaces $\ker T_j$ must be distinct.  This is sufficient for the following reason: suppose two kernels $\ker(T_r)$ and $\ker(T_s)$ coincide, where $r < s$.  Then 
\[ \mathbb{C}^n = \ker(T_n) = \begin{cases}
\ker\left( \prod_{i\leq r} \left( A^k -a_i I \right) \prod_{i>s}  \left( A^k -a_i I \right) \right), & s < n , \\
\ker\left( \prod_{i\leq r} \left( A^k -a_i I \right)\right),& s = n;
\end{cases}
\]
hence the minimal polynomial of $A^k$ has degree strictly less than $n$.  But this contradicts our assumption that $A$ is $k^{th}$--power maximal.  Therefore all kernels $\ker(T_j)$ must be distinct, and we must have proper containments 
\[ \left\{ \mathbf{0} \right\} \subsetneq \ker(T_1) \subsetneq \ker(T_2) \subsetneq \cdots \subsetneq \ker(T_n) = \mathbb{C}^n.\]
This guarantees that $\dim(\ker(T_j)) = j$, and since $W_j \subseteq \ker(T_j)$, we obtain the desired equalities $W_j = \ker(T_j)$.  Since each $W_j$ is the kernel of a polynomial in $A^k$, it follows that each $W_j$ is $A$--stable, hence $A$ is upper--triangular.

\end{proof}

Let $J$ be the matrix in Jordan canonical form 
        \setcounter{MaxMatrixCols}{20}
        \[ J = J(\lambda_1, s_1) \oplus 
             J(\lambda_2, s_2) \oplus 
          \cdots 
            \oplus  J(\lambda_k, s_k).
            \]
        Also let $M$ be the upper triangular matrix with the same diagonal entries as $J$, and entries $a_{i,j}, b_i$ on the superdiagonal given as follows:
        \[ M = \begin{bmatrix} 
                   \lambda_1  & a_{1,1}   \\
                     0 & \lambda_1 & \ddots \\
                    & & \ddots & a_{1, n_1-1} \\
                      0 & 0  &  \cdots & \lambda_1  & \mathbf{b_1} \\
          \vdots & \vdots && 0 &     \lambda_2  & a_{2,1}   \\
               &&&&      0 & \lambda_2 & \ddots \\
                &&&&    & & \ddots & a_{2, n_2-1} \\
                  &&&&    0 & 0  &  \cdots & \lambda_2 & \mathbf{b_2} \\
                   &&&&\vdots & \vdots && 0 &   \lambda_3  & a_{3,1}   \\
               &&&&&&&&      0 & \lambda_3 & \ddots \\
                &&&&&&&&    & & \ddots & a_{3, n_3-1} \\
                  & &&& &&&&   0 & 0  &  \cdots & \lambda_3 & \mathbf{b_3} \\
                    &&&& &&&& \vdots &  \vdots &  & & \ddots \\
        \end{bmatrix}
        \]
    \begin{lem}
    Let $J$ and $M$ be square matrices as given above.  If $J$ and $M$ satisfy the conditions 
    \[ M^k J M^{-k} = J^{\log_{\alpha}(1 - \alpha^k)} M J^{-\log_{\alpha}(1 - \alpha^k)}, \; \; \; \; 1 \leq k \leq q - 2, \]
    then all superdiagonal $a_{i,j} = 1$, $1 \leq j \leq n_i-1$, and all $b_i = 0$, $1 \leq i \leq k-1$.
        \end{lem}
    \begin{proof}
    Let $k$ be any positive integer.  First we show the product $M^k J M^{-k}$ is upper triangular, with the same diagonal entries as $M$ and superdiagonal entries given as follows: 
     \[ ( M^k J M^{-k})_{s, s+1} = \begin{cases}
            1, &  s \neq n_j  \\
            b_j\left( 1- \lambda_j^k / \lambda_{j+1}^k \right), & s = n_j
        \end{cases} \]
    \[ M^k J M^{-k} = \begin{bmatrix} 
                   \lambda_1  & 1   \\
                     0 & \lambda_1 & \ddots \\
                    & & \ddots & 1 \\
                      0 & 0  &  \cdots & \lambda_1  & b_1(1 - \frac{\lambda_1^k}{\lambda_2^k}) \\
          \vdots & \vdots &&&     \lambda_2  & 1  \\
               &&&&      0 & \lambda_2 & \ddots \\
                &&&&    & & \ddots & 1 \\
                  &&&&    0 & 0  &  \cdots & \lambda_2 & b_2(1 - \frac{\lambda_2^k}{\lambda_3^k}) \\
                   &&&&\vdots & \vdots &&&   \lambda_3  & 1  \\
               &&&&&&&&      0 & \lambda_3 & \ddots \\
                &&&&&&&&    & & \ddots & 1 \\
                  &&&& &&&&   0 & 0  &  \cdots & \lambda_3 & \\
                    &&&& &&&& & &  & & \ddots \\
        \end{bmatrix}
        \]
        Furthermore, the product $J^k M J^{-k}$ is also upper triangular, with the same diagonal entries as $J$ and superdiagonal entries given as follows: 
        \[ ( J^k M J^{-k})_{s, s+1} = \begin{cases}
            M_{s,s+1}, &  s \neq n_j  \\
            b_j\left( \lambda_j^k / \lambda_{j+1}^k \right), & s = n_j
        \end{cases} \]
        \[ J^k M J^{-k} = \begin{bmatrix} 
                   \lambda_1  & a_{1,1}   \\
                     0 & \lambda_1 & \ddots \\
                    & & \ddots & a_{1, n_1-1} \\
                      0 & 0  &  \cdots & \lambda_1  & \hspace{.2 in} b_1(\lambda_1^k / \lambda_2^k) \\
          \vdots & \vdots &&&     \lambda_2  & a_{2,1}  \\
               &&&&      0 & \lambda_2 & \ddots \\
                &&&&    & & \ddots & a_{2, n_2-1} \\
                  &&&&    0 & 0  &  \cdots & \lambda_2 & \hspace{-.2 in} b_2(\lambda_2^k / \lambda_3^k)) \\
                   &&&&\vdots & \vdots &&\hspace{.4 in} \ddots &   &  
        \end{bmatrix}
        \]
        For the first assertion: the base case $(k=1)$ can be verified by direct computation. For the inductive step, we calculate $M(M^k J M^{-k})M^{-1}$.  First we find $M(M^k J M^{-k})$ on the diagonal and superdiagonal entries.  The diagonal entries are $\lambda_1^2, \ldots, \lambda_k^2$ (with the same multiplicities that occur in $J$ and $M$).  The superdiagonal entries are, in order: 
        \[\lambda_1 (1 + a_{1,1}),\ldots, \lambda_1(1+a_{1, n_1-1}), \\
         b_1\left( \lambda_1 + \lambda_2  - \frac{\lambda_1^{k+1}}{\lambda_2^k}\right), \]
          \[\lambda_2 (1 + a_{2,1}),\ldots, \lambda_2(1+a_{2, n_2-1}), \\
         b_2\left( \lambda_2  +  \lambda_3  - \frac{\lambda_2^{k+1}}{\lambda_3^k}\right), \]
         \[ \vdots \]
The inverse of $M$ has the following diagonal and superdiagonal entries: 
\[  M^{-1} = \begin{bmatrix} 
                   \lambda_1^{-1}  & -a_{1,1} / \lambda_1^2   \\
                     0 & \lambda_1^{-1} & \ddots \\
                    & & \ddots & - a_{1 ,n_1-1} / \lambda_1^2  \\
                      0 & 0  &  \cdots & \lambda_1^{-1}  & - b_1 / \lambda_1 \lambda_2 \\
          \vdots & \vdots &&&     \lambda_2^{-1}  & - a_{2,1} / \lambda_2^2   \\
               &&&&      0 & \lambda_2^{-1} & \ddots \\
                &&&&    & & \ddots &  - a_{2, n_2-1} / \lambda_2^2 \\
                  &&&&    0 & 0  &  \cdots & \lambda_2^{-1} & \\
                   &&&&\vdots & \vdots && \hspace{0.4in} \ddots & &  \\
        \end{bmatrix}\]
    Multiplying $ M( M^k J M^{-k})$ with $M^{-1}$ yields the desired expressions on the diagonal and superdiagonal entries.  The proof for the second product  $J^k M J^{-k}$ uses a similar inductive argument.
    
    When we impose the conditions 
    \[ M^k J M^{-k} = J^{\log_{\alpha}(1 - \alpha^k)} M J^{-\log_{\alpha}(1 - \alpha^k)}, \; \; \; \; 1 \leq k \leq q - 2, \]
    we obain the following relations on the superdiagonal entries of $M$:
    \[ a_{ij} = 1, \text{ for all superdiagonal } a_{ij},\]
    and 
    \[ b_i \left( 1 - (\lambda_i / \lambda_{i+1})^k \right) = b_i\left( \lambda_i / \lambda_{i+1} \right)^{\log_{\alpha}(1 - \alpha^k)}, \; \; 1 \leq i \leq s-1, \; \; 1 \leq k \leq n-2.\]
    If some $b_i \neq 0$, then the ratio $r = \lambda_i / \lambda_{i+1}$ must satisfy all polynomials 
    \[ 1 - r^k = r^{\log_{\alpha}(1 - \alpha^k)}, \; \; 1 \leq k \leq q-2.\]
Lemma \ref{nonlinearsystem} prohibits any such solution, whence we conclude that all $b_i = 0$, and the Lemma is proven.
\end{proof}

    \begin{lem}\label{MJlemma}
        Let $J = J(\lambda_1, s_1) \oplus 
             J(\lambda_2, s_2) \oplus 
          \cdots 
            \oplus  J(\lambda_k, s_r)$.  Suppose $M$ is an upper triangular matrix with 
        \[ M_{i,i} = J_{i,i}, \;\; \; \; M_{i, i+1} = J_{i, i+1},\]
        so that $M$ and $J$ are equal on the diagonal and superdiagonal, and assume $M^k J M^{-k} = J^{\log_{\alpha}(1 - \alpha^k)} M J^{-\log_{\alpha}(1 - \alpha^k)}$ for all $ 1 \leq k \leq q-2$.  Then $M=J$.
    \end{lem}
    \begin{proof}
        We begin with the assertion that all elements along the 2nd superdiagonal $M_{i,i+2}$, $1 \leq i \leq n-2$, must be 0.  To show this, we first verify by induction on $k$ the equality 
        \[(M^k J M^{-k})_{i, i+2}  = M_{i, i+2}\left( 1 - M_{i,i}^k / M_{i+2, i+2}^k \right), \; \; 1 \leq i \leq n-2. \]
        Another inductive argument on $k$ shows that 
         \[(J^k M  J^{-k})_{i, i+2}  = M_{i, i+2}\cdot M_{i,i}^k / M_{i+2, i+2}^k, \; \; 1 \leq i \leq n-2.\]
         The assumption  $M^k J M^{-k} = J^{\log_{\alpha}(1 - \alpha^k)} M J^{-\log_{\alpha}(1 - \alpha^k)}$ for all $ 1 \leq k \leq q-2$ then gives us 
         \[  M_{i, i+2}\left( 1 - \left( M_{i,i} / M_{i+2, i+2} \right)^k \right) = M_{i, i+2}\cdot\left(  M_{i,i} / M_{i+2, i+2} \right)^{\log_{\alpha}(1 - \alpha^k)} , 1 \leq k \leq q-2.\]
         Let $r = M_{i,i} / M_{i+2, i+2}$.  If $M_{i, i+2} \neq 0$, this equation requires a nontrivial solution (in $r$) to the system 
         \[ 1 -  r^k = r^{\log_{\alpha}(1 - \alpha^k)}, \; \; 1 \leq k \leq q-2,\]
         which is prohibited by Lemma \ref{nonlinearsystem}.  Therefore we must have $M_{i, i+2} = 0.$

         An induction argument will now show that the entries of $M$ are 0 everywhere above the first superdiagonal.  For some $2 < r$, we assume that the $j^{th}$ superdiagonal entries $M_{i, i+j}$, $1 \leq i \leq n-j$, are all 0, for $1 \leq j \leq r$.  Another induction on $k$ shows that 
          \[  M_{i, i+r+1}\left( 1 - \left( M_{i,i} / M_{i+r+1, i+r+1} \right)^k \right) \]
          \[ \hspace{2 in} = M_{i, i+r+1}\cdot\left(  M_{i,i} / M_{i+r+1, i+r+1} \right)^{\log_{\alpha}(1 - \alpha^k)} , 1 \leq k \leq q-2,\]
          and Lemma \ref{nonlinearsystem} guarantees again that we must have $M_{i, i+r+1} = 0$ for all $1 \leq i \leq n-r$. Therefore the $(r+1)^{th}$ superdiagonal entries are all 0.  Since $M$ and $J$ are equal on their diagonals and superdiagonals, we have $M = J$.
    \end{proof}

\begin{pro}\label{cylic_corollary} Let $q$ be a prime power and $\alpha$ a primitive root in $\mathbb{F}_q$.  
    Let $J$ be an invertible $n \times n$ matrix, and suppose $J$ is $(q-1)^{th}$--power maximal.  If $M$ is an invertible $n \times n$ matrix satisfying 
    \[ \left[ J, M^{q-1} \right] = \left[ M, J^{q-1}\right] = 0,\]
    \[ M^k J M^{-k} = J^{\log_{\alpha}(1 - \alpha^k)} M J^{-\log_{\alpha}(1 - \alpha^k)}, \; \; \; \;  1 \leq k \leq q-2,\]
    then $J=M$.
\end{pro}
\begin{proof}
Without loss of generality we assume $J$ is in Jordan Normal form, hence upper triangular.  Then $J^{q-1}$ is also upper triangular.  Using the given defining relations, we then have  
\[ J^{q-1} = \left( M^k J M^{-k} \right)^{q-1} = \left( J^{\log_{\alpha}(1 - \alpha^k)} M J^{-\log_{\alpha}(1 - \alpha^k)} \right)^{q-1} = M^{q-1}.\]
Therefore $M^{q-1}$ is upper--triangular.  Since $M$ is similar to $J$, it also satisfies the condition that the characteristic polynomial of $M^{q-1}$ coincides with its minimal polynomial, hence $M$ is also $(q-1)^{th}$--power maximal, and so by Lemma \ref{upper_triangular} $M$ is also upper triangular.  Now Lemma \ref{MJlemma} gives us $J=M$. 

\end{proof}

        \subsection{Constant 2-dimensional representations of a cyclic quandle} Let $X = (\mathbb{Z}_q, \alpha)$ (so $q$ is a prime power and $\alpha$ is a primitive root of unity). Let $V = \mathbb{C}^2$.  Fix $a, b \in \mathbb{C}^*$, and let $\phi: X \rightarrow \text{Conj}(\text{Aut} (V))$ defined by the constant map 
        \[ \phi: x \mapsto A =  \begin{bmatrix} a & 0 \\ 0 & b \end{bmatrix}, \; \; \; x \in X. \]
        Then $\phi$ is a quandle morphism, because $\phi(0)$ and $\phi(1)$ both trivially satisfy the defining relations of $X$ (since $\left[ \phi(0), \phi(1) \right] = 0$).  Clearly $\mathbb{C} \begin{bmatrix} 1 \\ 0 \end{bmatrix}$ and $\mathbb{C} \begin{bmatrix} 0 \\ 1 \end{bmatrix}$ are 1-dimensional subrepresentations of $V$.  So this is a completely reducible 2-dimensional representation of $X$.

        Next, let $\psi: X \rightarrow \text{Conj}(\text{Aut} \mathbb{C}^2)$ defined by the constant map 
        \[ \psi: x \mapsto A =  \begin{bmatrix} \omega & 1 \\ 0 & \omega \end{bmatrix}, \; \; \; x \in X. \]
        Then $\phi$ is a quandle morphism, because $\psi(0)$ and $\psi(1)$ both trivially satisfy the defining relations of $X$ (since $\left[ \psi(0), \psi(1) \right] = 0$).  Then we can verify that $\mathbb{C} \begin{bmatrix} 1 \\ 0 \end{bmatrix}$ is a 1-dimensional subrepresentation of $\mathbb{C}^2$, while its vector space complement $\mathbb{C} \begin{bmatrix} 0 \\  1  \end{bmatrix}$ is \textit{not} a subrepresentation.  So this is a reducible - but not completely reducible - representation of $X$.
        \subsection{An explicit construction}
        Let $\pi$ denote the regular representation of the dihedral quandle $\Z_6$, $\pi: \Z_6 \rightarrow \Aut( \C\Z_6)$.  We use the quandle structure of $\Z_6$ to find the matrix representations of all elements of $\Z_6$.  We'll do this with two bases: first with the standard basis $\left\{ e_i \right\}_{i=0}^5$, and next with the basis taken from the decomposition of $\C\Z_6$ into irreducible representations.  
        
        With respect to the standard basis, we have 
        \[ \pi(0) = \pi(3)  =  \begin{bmatrix}  1&0&0&0&0&0 \\ 
        										0&0&0&0&0&1 \\
        										0&0&0&0&1&0	\\
        										0&0&0&1&0&0	\\
        										0&0&1&0&0&0 \\
        										0&1&0&0&0&0
        										\end{bmatrix}
         \] 
        \[ \pi(1) = \pi(4)  =  \begin{bmatrix}  0&0&1&0&0&0 \\ 
        										0&1&0&0&0&0 \\
        										1&0&0&0&0&0	\\
        										0&0&0&0&0&1	\\
        										0&0&0&0&1&0 \\
        										0&0&0&1&0&0
        										\end{bmatrix}
         \] 
        \[ \pi(2) = \pi(5)  =  \begin{bmatrix}  0&0&0&1&0&0 \\ 
        										0&0&1&0&0&0 \\
        										0&1&0&0&0&0	\\
        										1&0&0&0&0&1	\\
        										0&0&0&0&0&1 \\
        										0&0&0&0&1&0
        										\end{bmatrix}
         \] 
         
         Next we list the matrix representation with respect to the basis 
         \begin{align*} 
         v_1 & = (1,1,1,1,1,1), \\
         v_2 & = (1,-1,1,-1,1,-1), \\
         v_3 & = e_0 - e_2, \\
         v_4 & = e_2 - e_4,  \\
         v_5 & = e_1 - e_3, \\
         v_6 & = e_3 - e_5.
         \end{align*}
         Then, as described above, $\C \Z_6$ decomposes into four irreducible representations.

        \end{document}